\documentclass[12pt]{article}
\usepackage{amssymb,amsmath,amsthm}
\usepackage{graphicx, color}
\usepackage{array}
\usepackage{booktabs}
\usepackage{tikz}
\usepackage{times}
\usepackage{geometry}
\usepackage{setspace}
\usepackage{enumerate}

\newtheorem{theorem}{Theorem}
\newtheorem{lemma}[theorem]{Lemma}
\newtheorem{prop}[theorem]{Proposition}

\newtheorem{claim}[theorem]{Claim}

\newtheorem*{defn}{Definition}

\newtheorem{corollary}[theorem]{Corollary}

\theoremstyle{remark}

\newcommand{\R}{\mathbb R}

\newcommand{\sn}[1]{\| #1 \|^*}
\newcommand{\E}{\mathbb{E}}

\newcommand{\eps}{\epsilon}

\begin{document}
\title{Sign rank versus VC dimension\footnote{A preliminary version of this work was published in the proceeding of COLT'16.}}

\author{ Noga Alon\thanks{Sackler School of Mathematics and Blavatnik
School of Computer Science, Tel Aviv University, Tel Aviv 69978, Israel,
Microsoft Research, Herzliya, 
and School of Mathematics, Institute for Advanced Study, Princeton,
NJ 08540.  {\tt nogaa@tau.ac.il}.  Research supported in part by a
USA-Israeli BSF grant, by an ISF grant, by the Israeli I-Core program
and by the Oswald Veblen Fund.}
\and
Shay Moran\thanks{
Department of Computer Science, Technion-IIT, Microsoft Research, Herzliya,  and Max Planck Institute for Informatics, Saarbr\"{u}cken, Germany.
{\tt  shaymrn@cs.technion.ac.il}.}
\and
Amir Yehudayoff\thanks{
Department of Mathematics, Technion-IIT.
Email: \texttt{amir.yehudayoff@gmail.com}.
Horev fellow -- supported by the Taub
foundation.  Research also supported by ISF and BSF.}}

\maketitle
\begin{abstract}
This work studies the maximum
possible sign rank of $N \times N$ sign matrices with a given
VC dimension $d$.
For $d=1$, this maximum is {three}.
For $d=2$, this maximum is $\tilde{\Theta}(N^{1/2})$.
For $d >2$, similar but slightly less accurate statements hold.
{The lower bounds improve over previous ones by Ben-David et al., and 
the upper bounds are novel.}

The lower bounds are obtained by probabilistic constructions,
using a theorem of Warren in real algebraic topology.
The upper bounds are obtained using 
a result of Welzl about spanning trees with low stabbing number,
and using the moment curve.

The upper bound technique is also used to:
(i) provide estimates on the number of classes
of a given VC dimension, 
and the number of maximum 
classes of a given VC dimension -- answering a question of Frankl from '89, 
and (ii) design an efficient 
algorithm that provides an $O(N/\log(N))$ multiplicative
approximation for the sign rank.
%{(computing the sign rank is equivalent to the existential theory of the reals)}.

We also observe a general connection between sign rank
and spectral gaps which is based on Forster's argument. 
Consider the $N \times N$ adjacency matrix of a $\Delta$ regular graph
with a second eigenvalue of absolute value $\lambda$ and 
$\Delta \leq N/2$.
We show that the sign rank of the signed version of this matrix
is at least $\Delta/\lambda$.
We use this connection to prove the existence
of a maximum class $C\subseteq\{\pm 1\}^N$ with VC dimension $2$
and sign rank $\tilde{\Theta}(N^{1/2})$. This answers a question of 
Ben-David et al.~regarding the sign rank of large VC classes.
We also describe limitations of this approach,
in the spirit of the Alon-Boppana theorem.

We further describe connections to communication complexity, geometry,
learning theory, and combinatorics.
\end{abstract}

%\begin{keywords}
%VC dimension, dimension complexity, sign rank, kernel machines, maximum classes,
%spectral gap, communication complexity. 
%\end{keywords}

\thispagestyle{empty}
\newpage

\setcounter{page}{0}

\setcounter{tocdepth}{2}
\tableofcontents

\thispagestyle{empty}

\setcounter{page}{0}
\newpage

\section{Introduction}
Boolean matrices (with $0,1$ entries)
and sign matrices (with $\pm 1$ entries)
naturally appear in many areas of research\footnote{There is a standard transformation of a boolean matrix $B$ to the sign matrix $S= 2B - J$, where
$J$ is the all $1$ matrix.
The matrix $S$ is called the signed version of $B$,
and the matrix $B$ is called the boolean version of $S$.}.
We use them e.g.\ to represent set systems and graphs in combinatorics,
hypothesis classes in learning theory,
and boolean functions in communication complexity.

This work further investigates the relation between two useful complexity measures on sign matrices.

\begin{defn}[Sign rank]
For a real matrix $M$ with no zero entries, let $\text{sign}(M)$ denote the sign matrix such
that $(\text{sign}(M))_{i,j}=\text{sign}(M_{i,j})$ for all $i,j$.
The sign rank of a sign matrix $S$ is defined as
$$\text{sign-rank}(S) = \min \{ \text{rank}(M) : \text{sign}(M)=S\},$$
where the rank is over the real numbers.
It captures the minimum dimension of a real space
in which the matrix can be embedded using half spaces 
%NA
through the origin \footnote{That is, the columns correspond to points in $\R^k$ and
the rows to half spaces 
%NA
through the origin
(i.e.\ collections of all points $x \in \R^k$ so 
that $\langle x, v \rangle \geq 0$
for some fixed $v \in \R^k$).} 
%that $\langle x, v \rangle \geq \theta$
%for some fixed $v \in \R^k$ and $\theta \in \R$).} 
(see for example~\cite{Lokam}).
\end{defn}

\begin{defn}[Vapnik-Chervonenkis dimension]
The VC dimension of a sign matrix $S$, denoted $VC(S)$, is defined as follows.
A subset $C$ of the columns of $S$ is called shattered 
if each of the $2^{|C|}$ different patterns of ones and minus ones appears
in some row in the restriction of $S$ to the columns in $C$.
The VC dimension of $S$ is the maximum size of
a shattered subset of columns. It captures 
the size of the minimum $\epsilon$-net 
for the underlying set system~\cite{HW86,KPW92}.
\end{defn}

The VC dimension and the sign rank appear in various areas % of research
of computer science and mathematics.
One important example is learning theory, where 
the VC dimension captures the sample complexity of learning in the PAC model~\cite{DBLP:conf/stoc/BlumerEHW86,VCpaper}, and the sign rank 
relates 
to the {generalization guarantees} of practical learning algorithms,
such as support vector machines, large margin classifiers, and kernel classifiers~\cite{DBLP:conf/coco/LinialS08,
DBLP:conf/fsttcs/ForsterKLMSS01,DBLP:conf/colt/ForsterSS01, 
DBLP:journals/tcs/ForsterS06,DBLP:journals/datamine/Burges98,
DBLP:books/daglib/0097035}.
Loosely speaking, the VC dimension relates to learnability, while sign rank
relates to learnability by linear classifiers.
Another example is communication complexity, where
the sign rank is equivalent to the unbounded error randomized communication
complexity~\cite{DBLP:journals/jcss/PaturiS86}, and the VC dimension
relates to one round distributional communication complexity
under product distributions~\cite{DBLP:conf/stoc/KremerNR95},

The main focus of this work is
how large can the sign rank be for a given VC dimension.
In learning theory, this question concerns the universality of linear classifiers.
In communication complexity, 
this concerns the difference between randomized 
communication complexity with unbounded error
and between communication complexity under product distribution
with bounded error.
Previous works have studied these differences 
from the communication complexity perspective~\cite{DBLP:journals/cc/Sherstov10,DBLP:conf/coco/Sherstov07} and the learning theory perspective~\cite{DBLP:journals/jmlr/Ben-DavidES02}.
In this work we provide explicit matrices and 
stronger separations compared 
to those of~\cite{DBLP:journals/cc/Sherstov10,DBLP:conf/coco/Sherstov07}
and~\cite{DBLP:journals/jmlr/Ben-DavidES02}. See the discussions in Section~\ref{sec:mainseparation} and Section~\ref{sec:communication} for more details.

\subsection{Duality}

We start by providing alternative descriptions of the VC dimension and sign rank,
which demonstrate that these notions are dual to each other.
The sign rank of a sign matrix $S$ is the maximum number $k$ such that
\begin{align*}
& \forall \ M\mbox{ such that } \text{sign}(M)=S \ \ \exists\mbox{ $k$ columns
$j_1,\ldots,j_k$ } \\
& \qquad  \mbox{the columns $j_1,\ldots,j_k$ are linearly independent in $M$}
\end{align*}
The {\em dual sign rank} of $S$ is the maximum number $k$ such that
\begin{align*}
& \exists\mbox{ $k$ columns $j_1,\ldots,j_k$} 
 \ \  \forall \ M\mbox{ such that } \text{sign}(M)=S \\
  & \qquad  \mbox{the columns $j_1,\ldots,j_k$ are linearly independent in $M$}.
\end{align*}
%Consider the following definition that is obtained by flipping the order of the quantifiers.
It turns out that the dual sign rank is 
almost equivalent to the VC dimension
(the proof is in Section~\ref{sec:duality}).
\begin{prop}\label{prop:dualsignrank}
$VC(S)\leq \text{dual-sign-rank}(S) \leq 2VC(S)+1$.
\end{prop}

As the dual sign rank is at most the sign rank,
it follows that
the VC dimension is at most the sign rank.
This provides further motivation for studying
the largest possible gap between sign rank and VC dimension;
it is equivalent to the largest possible gap
between the sign rank and the dual sign rank. 

It is worth noting that there are some interesting classes of matrices for which these quantities are equal.
One such example is the $2^n\times 2^n$ disjointness matrix $DISJ$, whose rows and columns are indexed by all subsets of $[n]$, and $DISJ_{x,y}=1$ if and only if $|x\cap y|  > 0$.
For this matrix both the sign rank and the dual sign rank are exactly $n+1$.

\subsection{Sign rank versus VC dimension}\label{sec:mainseparation}

The VC dimension is at most the sign rank. 
On the other hand, it is long known that the sign rank is not bounded from above
by any function of the VC dimension.
Alon, Haussler, and Welzl \cite{DBLP:conf/compgeom/AlonHW87}
provided examples of $N\times N$ matrices with VC dimension $2$
for which the sign rank 
tends to infinity with $N$.
\cite{DBLP:journals/jmlr/Ben-DavidES02}
used ideas from \cite{DBLP:conf/focs/AlonFR85} together
with estimates concerning the Zarankiewicz problem
to show that many matrices with constant VC dimension
(at least $4$) have high sign rank.

We further investigate the problem of determining or estimating the
maximum possible sign rank of $N \times N$ matrices with VC dimension
$d$. Denote this maximum by $f(N,d)$.
We are mostly interested in fixed $d$ and $N$ tending to infinity.

We observe that there is a dichotomy between the behaviour of $f(N,d)$ when $d=1$
and when $d>1$.
The value of $f(N,1)$ is $3$,
but for $d>1$, the value of $f(N,d)$ tends to infinity with $N$.
We now discuss the behaviour of $f(N,d)$ in more detail,
and describe our results.

We start with the case $d=1$.
The following theorem and claim imply that 
for all $N \geq 4$,
$$f(N,1)=3.$$

The following theorem which 
was proved by \cite{DBLP:conf/compgeom/AlonHW87}
shows that for $d=1$,
matrices with high sign rank do not exist.
For completeness, we provide our simple 
and constructive proof in Section~\ref{sec:VCone}.

\begin{theorem}[\cite{DBLP:conf/compgeom/AlonHW87}]
\label{thm:VC1}
If the VC dimension of a sign matrix $M$
is one then its sign rank is at most $3$.
\end{theorem}

We also note that the bound $3$ is tight
(see Section~\ref{sec:VCone} for a proof).
\begin{claim}
\label{clm:Id4}
For $N\geq 4$, the $N \times N$ signed identity matrix
(i.e.\ the matrix 
with $1$ on the diagonal and~$-1$ off the diagonal) 
has VC dimension one and sign rank $3$.
\end{claim}

Next, we consider the case $d>1$,
starting with lower bounds on $f(N,d)$.
As mentioned above, two lower bounds were previously known: 
\cite{DBLP:conf/compgeom/AlonHW87} showed that
$f(N,2) \geq \Omega(\log N)$.
\cite{DBLP:journals/jmlr/Ben-DavidES02} showed that 
$f(N,d) \geq \omega(N^{1-\frac{2}{d}-\frac{1}{2^{d/2}}})$, for every fixed $d$, 
which provides a nontrivial result only for $d\geq4$. 
We prove the following stronger lower bound.
\begin{theorem}
\label{t42}
The following lower  bounds on $f(N,d)$ hold:
\begin{enumerate}
\item
$f(N,2) \geq \Omega(N^{1/2}/ \log N)$.
\item
$f(N,3) \geq \Omega(N^{8/15}/\log N)$.
\item
$f(N,4) \geq \Omega(N^{2/3}/\log N)$.
\item
For every fixed $d>4$,
$$
f(N,d) \geq \Omega(N^{1-(d^2+5d+2)/(d^3+2d^2+3d)}/ \log N).
$$
\end{enumerate}
\end{theorem}

To understand part 4 better, notice that
$$\frac{d^2+5d+2}{d^3+2d^2+3d}
= \frac{1}{d} + \frac{3d-1}{d^3+2d^2+3d},$$
which is close to $1/d$ for large $d$.
The proofs are described in Section~\ref{sec:VCdim},
where we also discuss the tightness of our arguments.

What about upper bounds on $f(N,d)$?
It is shown in
\cite{DBLP:journals/jmlr/Ben-DavidES02} that for every matrix in a
certain class of $N \times N$ matrices with constant VC dimension,
the sign rank is at most $O(N^{1/2})$. The proof uses the
connection between sign rank and communication complexity.
However, there is no general upper bound for the sign rank of
matrices of VC dimension $d$ in 
\cite{DBLP:journals/jmlr/Ben-DavidES02}, and the authors 
explicitly mention the absence of such a result.

Here we prove the following upper bounds,
using a concrete embedding 
of matrices with low VC dimension in real space.
\begin{theorem}
\label{t41}
For every fixed $d \geq 2$,
$$
f(N,d) \leq O(N^{1-1/d}).
$$
\end{theorem}

In particular, this determines 
$f(N,2)$ up to a logarithmic factor:
$$
\Omega(N^{1/2}/ \log N) \leq f(N,2) \leq O(N^{1/2}).
$$

The above results imply existence of sign matrices
with high sign rank. However, their proofs use
counting arguments and hence do not provide a method
of certifying high sign rank for explicit matrices.
In the next section we show how one 
can derive a lower bound for the sign rank of many explicit matrices.

\subsection{Sign rank and spectral gaps}

Spectral properties of boolean matrices are known to be deeply 
related to their combinatorial structure.
Perhaps the best example is Cheeger's inequality
which relates spectral gaps to combinatorial expansion
\cite{zbMATH03808175,AM0,zbMATH03875335,Al1,zbMATH05302790}.
Here, we describe connections between spectral properties
of boolean matrices and the sign rank of their signed versions.

Proving strong lower bounds on the sign rank of sign matrices
turned out to be a difficult task.
Alon, Frankl, and  R{\"{o}}dl~\cite{DBLP:conf/focs/AlonFR85} 
were the first to prove that 
there are sign matrices with high sign rank,
but they have not provided explicit examples.
Later on, a breakthrough of \cite{DBLP:conf/coco/Forster01}
showed how to prove lower bounds on the sign rank of explicit matrices,
proving, specifically, that Hadamard matrices have high sign rank.
\cite{DBLP:conf/focs/RazborovS08} proved that there is a function 
that is computed by a small depth three boolean circuit, 
but with high sign rank.
It is worth mentioning that no explicit matrix whose sign
rank is significantly larger
than $N^{\frac{1}{2}}$ is known.

We focus on the case of regular matrices,
but a similar discussion can be carried more generally.
A boolean matrix is $\Delta$ regular 
if every row and every column in it has exactly
$\Delta$ ones, and a sign matrix is $\Delta$ regular if
its boolean version is $\Delta$ regular.

An $N \times N$ real matrix $M$ has $N$ singular values
$\sigma_1 \geq \sigma_2 \geq \ldots 
\geq \sigma_N \geq 0$.
The largest singular value of $M$
is also called its spectral norm
$\|M\| = \sigma_1 = \max \{ \|M x\| : \|x\| \leq 1\},$
where $\|x\|^2 = \langle x , x\rangle$
with the standard inner product.
%The second largest singular value of $M$ is denoted here by
%$\sigma(M) = \sigma_2.$
If the ratio $\sigma_2(M)/\|M\|$ is bounded away from one,
or small, we say that $M$ has a spectral gap.

We prove that if $B$ has a spectral gap
then the sign rank of $S$ is high. 
\begin{theorem}
\label{thm:SpecVsSignRank}
Let $B$ be a $\Delta$ regular $N \times N$ boolean matrix
with $\Delta \leq N/2$, and let $S$ be its signed version.
Then,
$$\text{sign-rank}(S) \geq \frac{\Delta}{\sigma_2(B)}.$$
\end{theorem}

In many cases a spectral gap for $B$ implies that
it has pseudorandom properties.
This theorem is another manifestation of this phenomenon
since random sign matrices have high sign rank (see~\cite{DBLP:conf/focs/AlonFR85}).

The theorem above provides a non trivial lower bound on the sign rank of $S$.
There is a non trivial upper bound as well.
The sign rank of a $\Delta$ regular sign matrix is at most $2\Delta+1$.
Here is a brief explanation of this upper bound
(see \cite{DBLP:conf/focs/AlonFR85} for a more detailed proof).
Every row $i$ in $S$ has at most $2 \Delta$
sign changes (i.e.\ columns $j$ so that $S_{i,j} \neq S_{i,j+1}$).
This implies that for every $i$,
there is a real univariate polynomial $G_i$ of degree 
at most $2\Delta$
so that $G_i(j) S_{i,j} > 0$ for all $j \in [N] \subset \R$.
To see how this corresponds to sign rank at most $2\Delta+1$, recall that
evaluating a polynomial $G$ of degree $2\Delta$ on a point $x \in \R$
corresponds to an inner product 
over $\R^{2\Delta +1}$ between the vector of coefficients of $G$,
and the vector of powers of $x$.

Our proof of Theorem~\ref{thm:SpecVsSignRank}
and its limitations are discussed in detail
in Section~\ref{sec:LBproof}.

\section{Applications}
\subsection{Learning theory}
%%%%NEW

\subsubsection*{Universality of linear classifiers}
{Linear classifiers} have been central in the study of machine learning since the introduction of the Perceptron algorithm in the 50's~\cite{perceptron} and Support Vector Machines (SVM) in the 90's~\cite{DBLP:conf/colt/BoserGV92,DBLP:journals/ml/CortesV95}.
The rising of kernel methods in the 90's~\cite{DBLP:conf/colt/BoserGV92,DBLP:journals/neco/ScholkopfSM98}
enabled reducing many learning problems to the framework of halfspaces, making linear classifiers a central algorithmic tool.

These methods use the following two-step approach. First,
embed the hypothesis class\footnote{In this context we use the more common term ``hypothesis class'' instead of ``matrix.''} in halfspaces of an Euclidean space (each point corresponds to a vector and for every hypothesis $h$, the vectors corresponding to $h^{-1}(1)$ and the vectors corresponding to $h^{-1}(-1)$ are separated by a hyperplane). Second, apply a learning algorithm for halfspaces.  

If the embedding is to a low dimensional space then a good generalization rate is implied.
For embeddings to large dimensional spaces, SVM theory offers an alternative parameter, namely the margin\footnote{The margin of the embedding is the minimum over all hypotheses $h$ of the distance between the convex hull of the vectors corresponding to $h^{-1}(1)$ and the convex hull of the vectors corresponding to $h^{-1}(-1)$}. Indeed, a large margin also implies a good generalization rate.
On the other hand, any embedding with a large margin can be projected to a low dimensional space using standard dimension reduction arguments~\cite{JL,DBLP:conf/focs/ArriagaV99,DBLP:journals/jmlr/Ben-DavidES02}.
%To conclude, if $C$ is learned by a kernel based algorithm then its sign-rank is low.

Ben-David, Eiron, and Simon~\cite{DBLP:journals/jmlr/Ben-DavidES02} utilized it to argue that
``\ldots any universal learning machine, which transforms data to a Euclidean space and then applies linear (or large margin) classification, cannot preserve good generalization bounds in general.'' Formally, they showed that: 
For any fixed $d>1$, most hypothesis classes 
$C\subseteq\{\pm 1\}^N$ of VC dimension $d$ have sign-rank of $N^{\Omega(1)}$.
As discussed in Section~\ref{sec:mainseparation}, 
Theorem~\ref{t42} quantitatively improves over their results.

In practice, linear classifiers are widely used in a variety 
of applications including handwriting recognition, 
image classification, medical science, bioinformatics, and more. 
The practical usefulness of linear classifiers and the argument
of Ben-David, Eiron, and Simon manifest a gap between practice and theory that seems worth studying.
We next discuss how Theorem~\ref{t41}, which provides a non-trivial upper bound on the sign rank,
can be interpreted as a theoretical evidence which supports the practical usefulness of linear classifiers.
Let $C\subseteq\{\pm 1\}^X$ be a hypothesis class, and let $\gamma>0$.
We say that $C$ is {\em $\gamma$-weakly represented by halfspaces}
if for every finite $Y\subseteq X$, the sign rank
of $C|_Y$ is at most $O(|Y|^{1-\gamma})$.
In other words, there exists an embedding of $Y$
in $\mathbb{R}^k$ with $k=O(|Y|^{1-\gamma})$ such that each hypothesis
in $C|_Y$ corresponds to a halfspace in the embedding.
Theorem~\ref{t41} shows that any class $C$ 
is $\gamma$-weakly represented by halfspaces where $\gamma$ depends only on its VC dimension.
{Weak representations can be thought of as providing a compressed
representation of $C|_Y$ using half-spaces in a dimension that is sublinear in $|Y|$.
Such representations imply learnability;}
indeed, every $\gamma$-weakly represented class $C$ is learnable,
as the VC dimension of $C$ is bounded from above by some function of of $\gamma$.
%indeed, it is not hard to see that its VC dimension is at most some $O(({1}/{\gamma})^{({1}/{\gamma})})$.
While these quantitative relations between the VC dimension and $\gamma$ may be rather loose, they show that in principle, any learnable class has a weak representation by halfspaces which certifies its learnability. 
%Perhaps, in practice, kernel machines find weak representations even when the sign rank is large.

\subsubsection*{Maximum classes with large sign rank}
Let $C\subseteq\{\pm 1\}^N$ be a class with VC dimension $d$.
The class $C$ is called maximum if it meets the Sauer-Shelah's bound~\cite{Sa} with equality\footnote{Maximum classes are distinguished from maximal classes: A maximum class has the largest possible size among all classes of VC dimension $d$, and a maximal class is such that for every sign vector $v\notin C$, 
if $v$ is added to $C$ then the VC dimension is increased.}. That is, $|C|=\sum_{i=0}^{d}{N \choose i}$.  Maximum classes were studied in different contexts such as 
machine learning, geometry, and combinatorics (e.g.~\cite{BR95,FW95,Welzl,Dress,ShatNews,KW07,Mor12,RR3,RRB14}).

There are several known examples of maximum classes. A fairly simple one
is the hamming ball of radius $d$, i.e., the class of all vectors with weight at most $d$. Another set of examples relates to the sign rank:
Let $H$ an arrangement of hyperplanes in $\mathbb{R}^d$.
These hyperplanes cut $\mathbb{R}^d$ into cells; the connected components of $\mathbb{R}^d\setminus \left(\bigcup_{h\in H}{h}\right)$. Each cell $c$ is associated with a sign vector
$v_c\in \{\pm 1\}^H$ which describes the location of the cell relative to each of the hyperplanes.
%: $v_c(h)=1$ if $h$ lies on the positive side of $h$ and $v_c(h)=-1$
%if $h$ lies on the negative side of $h$.
See Figure~\ref{fig:hyper} for a planar arrangement.
The sign rank of such a class is at most $d+1$. 
It is known (see e.g.~\cite{Welzl}) that if the hyperplanes are in general position then the sign vectors of the cells form a maximum class of VC dimension $d$. 

\begin{figure}
\label{fig:hyper}
\begin{center}
\includegraphics[width=.33\textwidth]{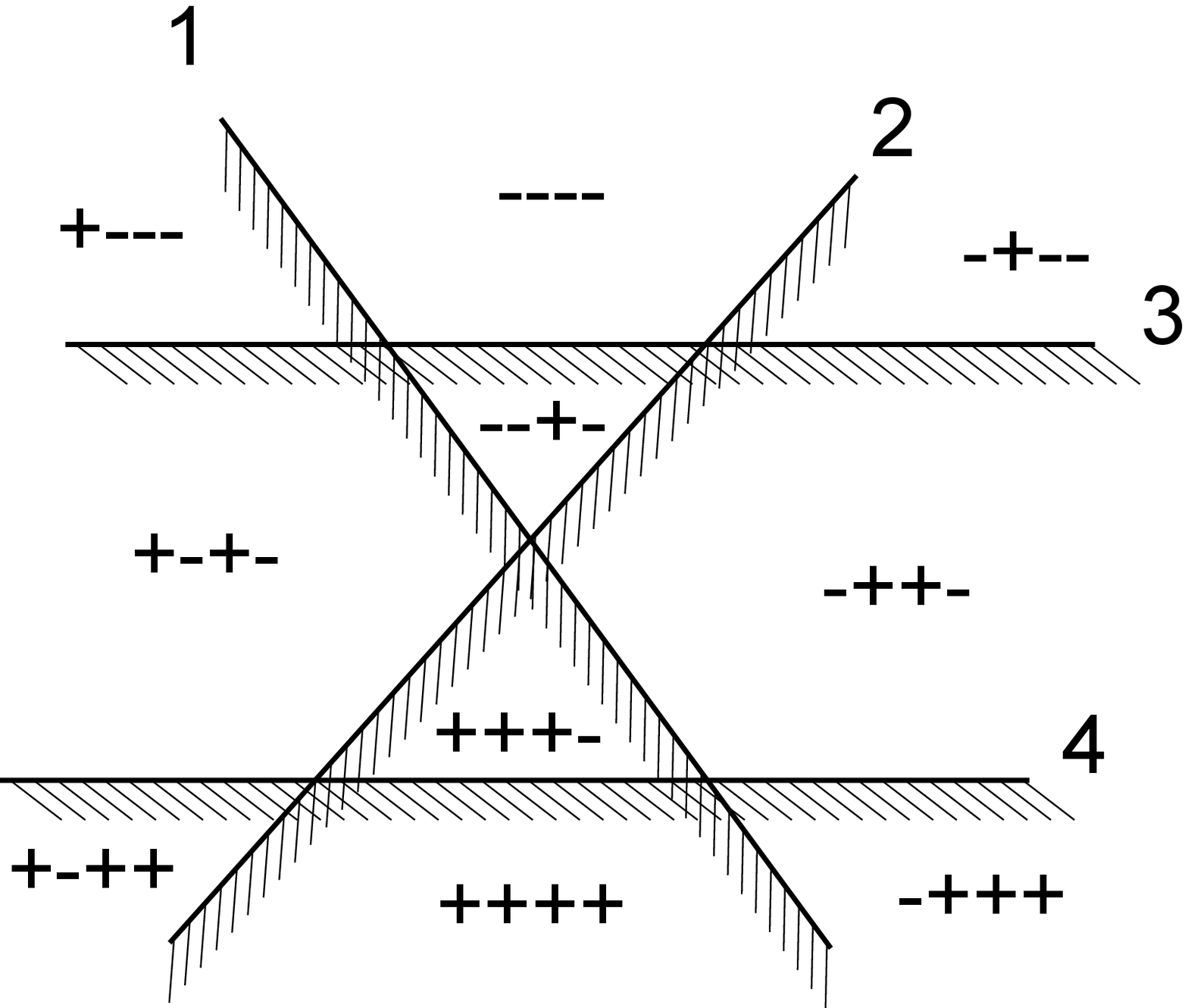}
\end{center}
\caption{An arrangement of lines in the plane and the corresponding cells.}
\end{figure}

G{\"{a}}rtner and Welzl~\cite{Welzl} gave a combinatorial characterization of  maximum classes constructed using generic halfspaces. As an application of their characterization they note that hamming ball of radius $d$ is a maximum class that can not be realized this way. By Lemma~\ref{l43}, however, the hamming ball of radius $d$ has sign rank at most $2d+1$
(it is in fact exactly $2d+1$).
It is therefore natural to ask whether every maximum class 
has sign rank which depends only on $d$. A similar question was also asked by~\cite{DBLP:journals/jmlr/Ben-DavidES02}.  
Theorem~\ref{thm:projint} in Section~\ref{sec:maximum} gives a negative answer to this question, even when $d=2$ (when $d=1$, by Theorem~\ref{thm:VC1} the sign rank is at most $3$).

In machine learning, maximum classes were studied extensively in the context of sample compression schemes.
A partial list of works in this context includes~\cite{FW95,KW07,RR3,RRB14,MW15,DBLP:conf/alt/DoliwaSZ10}.
\cite{RR3} constructed an unlabeled sample compression scheme for maximum classes.
Their scheme uses an approach suggested by~\cite{KW07} and their analysis 
resolved a conjecture from~\cite{KW07}.
A crucial part in their work is establishing the existence of an embedding of any maximum class of VC dimension $d$ in an arrangement of piecewise-linear hyperplanes in $\mathbb{R}^d$. Theorem~\ref{thm:projint} below shows that even for VC dimension $2$, there are maximum classes $C\subseteq\{\pm 1\}^N$ of sign rank $\Omega(N^{1/2}/\log{N})$. Thus, in order to make the piecewise-linear arrangement in $\mathbb{R}^2$ linear the dimension of the space must significantly grow to  $\Omega(N^{1/2}/\log{N})$.

\subsection{Explicit examples}
\label{sec:IntroFinGeo}
The spectral lower bound on sign rank 
gives many explicit examples of
matrices with high sign rank, which come from
known constructions of expander graphs and combinatorial designs. 
A rather simple such family of examples
is finite projective geometries.

Let $d \geq 2$ and $n \geq 3$.
Let $P$ be the set of points in a $d$ dimensional projective
space of order $n$, and let $H$ be the set of hyperplanes in the space.
For $d=2$, this is just a projective plane with points and lines.
It is known (see, e.g., \cite{zbMATH01382771}) that
$$|P|=|H|=N_{n,d} := n^d+n^{d-1}+ \ldots +n+1 = \frac{n^{d+1}-1}{n-1}.$$ 
Let $A\in\{\pm 1\}^{P \times H}$ be the signed point-hyperplane incidence matrix:
$$A_{p,h} = \left\{
\begin{array}{ll}
1 & p  \in h ,\\
-1 & p \not \in h .\\
\end{array} \right.$$

\begin{theorem}
\label{thm:projective}
The matrix $A$ is $N \times N$ 
with $N = N_{n,d}$, its VC dimension is $d$,
and its sign rank is larger than 
$$\frac{n^d -1}{n^{\frac{d-1}{2}}(n-1)} 
\geq N^{\frac{1}{2}-\frac{1}{2d}}.$$
\end{theorem}

The theorem follows from known properties
of projective spaces (see Section~\ref{sec:PS}).
A slightly weaker (but asymptotically equivalent) lower bound on the sign rank of $A$ was given by~\cite{DBLP:conf/fsttcs/ForsterKLMSS01}.

The sign rank of $A$ is at most $2N_{n,d-1}+1=O(N^{1-\frac{1}{d}})$, due to the observation in \cite{DBLP:conf/focs/AlonFR85} mentioned above. To see this, note that every point in the projective space is incident to $N_{n,d-1}$ hyperplanes. 

Other explicit examples come from spectral graph theory.
Here is a brief description of matrices 
that are even more restricted than having VC dimension $2$
but have high sign rank; 
no $3$ columns in them have more than $6$ distinct projections. 
An $(N,\Delta, \lambda)$-graph is a $\Delta$ regular graph on
$N$ vertices so that the absolute value of every eigenvalue of
the graph besides the top one is at most $\lambda$. There are 
several known constructions of $(N,\Delta,\lambda)$-graphs for which
$\lambda \leq O(\sqrt {\Delta})$, that do not contain short
cycles. Any such graph with $\Delta \geq N^{\Omega(1)}$ provides
an example with sign rank at least $N^{\Omega(1)}$, and if there is
no cycle of length at most $6$ then in the sign matrix we have at
most $6$ distinct projections on any set of $3$ columns.

\subsubsection{Maximum classes}\label{sec:maximum}

Let $P$ be the set of points in a projective plane of order $n$ and 
let $L$ be the set of lines in it. Let $N=N_{n,2}=|P|=|L|$.
For each line $\ell \in L$, fix some linear order on the points in $\ell$. 
A set $T\subset P$ is called an interval if $T\subseteq\ell$ for some line
$\ell\in L$, and $T$ forms an interval with respect to the order we fixed on $\ell$.

\begin{theorem}\label{thm:projint}
The class
$R$ of all intervals is a maximum class of VC dimension $2$.
Moreover, there exists a choice of linear orders for the lines in $L$
such that the resulting $R$ has sign rank $\Omega(N^{1/2}/ \log N)$.
\end{theorem}

The proof of Theorem~\ref{thm:projint} is given in Section~\ref{sec:PS}.
The proof does not follow directly from Theorem~\ref{t42}
since it is not clear that the classes with VC dimension $2$ and large sign rank which
are guaranteed to exist by Theorem~\ref{t42} can be extended to a maximum class.

\subsection{Computing the sign rank}

Linear Programming (LP) is one of the most famous and useful problems in the class P.
As a decision problem, an LP problem concerns determining the satisfiability of a system
$$\ell_i(x) \geq 0,~i=1,\ldots,m$$
where each $\ell_i$ is an affine function defined over $\mathbb{R}^n$ (say with integer coefficients). 
A natural extension of LP is to consider the case in which each $\ell_i$ is a multivariate polynomial. 
Perhaps not surprisingly, this problem is much harder than LP.
In fact, satisfiability of a system of polynomial inequalities is known to be a complete problem for the class $\exists\mathbb{R}$. The class $\exists\mathbb{R}$ is known to lie between PSPACE and NP (see~\cite{Matousek} and references within).

Consider the problem of deciding whether the sign rank of a given $N\times N$ sign matrix is at most $k$.
A simple reduction shows that to solve this problem it is enough to decide whether
a system of real polynomial inequalities is satisfiable. 
Thus, this problem belongs to the class $\exists\mathbb{R}$.
~\cite{BFJK09}\footnote{Interestingly,
their motivation for considering sign rank
comes from image processing.}, and~\cite{BK13} showed that deciding if the sign rank is at most $3$ is NP-hard, and that deciding if the sign rank is at most $2$ is in P.
Both~\cite{BFJK09}, and~\cite{BK13} established the NP-hardness of deciding whether the sign-rank is at most $3$
by a reduction from the problem of determining stretchacility of pseudo-line arrangements.
This problem concerns whether a given combinatorial description of an arrangement of pseudo-lines
can be realized (``stretched'') by an arrangement of lines.
\cite{Matousek}, based on the works of \cite{Mne89}, \cite{Sho91}, and~\cite{RG95}
showed that determining stretchability of pseudo-line arrangements is in fact $\exists\mathbb{R}$-complete.
Therefore, it follows\footnote{\cite{Matousek} considers a different type of combinatorial description than~\cite{BFJK09,BK13}, and therefore considered a different formulation of the stretchability problem. However, it is possible to transform between these descriptions in polynomial time.} that determining whether the sign-rank is at most $3$ is $\exists\mathbb{R}$-complete.

%Therefore, a result of Canny~\cite{Canny} implies that this problem
%is in PSPACE. It is still open whether this problem is in NP, but

Another related work of~\cite{LeeS09} concerns 
the problem of computing the approximate rank of a
sign matrix, for which they provide an
approximation algorithm. 
They pose the problem of efficiently approximating the
sign rank as an open problem.

Using an idea similar to the one in the proof of Theorem~\ref{t41}
we derive an approximation algorithm for the sign rank
(see Section~\ref{sec:approxproof}).
\begin{theorem}\label{t43}
There exists a polynomial time algorithm that approximates the 
sign rank 
%NA
of a given $N$ by $N$ matrix up to a
multiplicative factor of $c \cdot N / \log(N)$
where $c>0$ is a universal constant.
\end{theorem}

\subsection{Communication complexity}\label{sec:communication}
We briefly explain the notions from communication complexity we use.
For formal definitions, background and more details, 
see the textbook \cite{DBLP:books/daglib/0011756}.

For a function $f$ and a distribution $\mu$
on its inputs, 
define $D_\mu(f)$ as the minimum communication complexity of a 
%deterministic\footnote{In the distributional setting,
%every randomized protocol for $f$ can be replaced
%by a deterministic 
%protocol for $f$ without increasing the error nor the communication.}
protocol that 
correctly computes $f$ with error $1/3$ over inputs from $\mu$.
Define
$D^\times(f) = \max \{D_\mu(f) : 
\text{$\mu$ is a product distribution}\}.$
Define the unbounded error communication complexity $U(f)$ of $f$
as the minimum communication complexity of a randomized private-coin\footnote{In the public-coin model, every boolean function has
unbounded communication complexity at most two.} protocol that correctly
computes $f$ with probability strictly larger than $1/2$ on every input.

Two works of~\cite{DBLP:journals/cc/Sherstov10,DBLP:conf/coco/Sherstov07}
showed that there are functions with small 
distributional communication complexity
under product distributions,
and large unbounded error communication complexity.
In \cite{DBLP:journals/cc/Sherstov10} the separation
is as strong as possible but it is not for an explicit function,
and the separation in \cite{DBLP:conf/coco/Sherstov07}
is not as strong but the underlying function is explicit.

The matrix $A$ with $d=2$ and $n\geq 3$ in our example 
from Section~\ref{sec:IntroFinGeo}
corresponds to the 
following communication problem:
Alice gets a point $p \in P$,
Bob gets a line $\ell \in L$,
and they wish to decide whether $p \in \ell$ or not.
Let $f: P \times L\rightarrow\{0,1\}$ be the corresponding
function and let $m = \lceil \log_2(N) \rceil$. 
A trivial protocol would
be that Alice sends Bob using $m$ bits the name of her point, 
Bob checks whether it is incident to the line, and outputs accordingly.

Theorem~\ref{thm:projective} implies the following consequences.
Even if we consider protocols that use randomness and are allowed 
to err with probability less than but arbitrarily close to $\frac{1}{2}$,
then still one cannot do 
considerably better than the above trivial protocol.
However, if the input $(p,\ell)\in P\times L$ is distributed according to 
a product distribution then there exists an $O(1)$ protocol that errs
with probability at most $\frac{1}{3}$.

\begin{corollary}
The unbounded error communication complexity of
$f$ is\footnote{By taking larger values of $d$,
the constant $\frac{1}{4}$ may be increased to $\frac{1}{2}-\frac{1}{2d}$.}
$U(f) \geq \frac{m}{4} - O(1)$.
The distributional communication complexity of $f$ under product distributions
is 
$D^\times(f) \leq O(1)$.
\end{corollary}

These two seemingly contradicting facts are a corollary
of the high sign rank and the low VC dimension of $A$,
using two known results.
The upper bound on $D^\times(f)$ follows from
the fact that $\text{VCdim}(A)=2$, and
the work of~\cite{DBLP:conf/stoc/KremerNR95}
which used the PAC learning algorithm
to construct an efficient (one round) communication protocol for $f$
under product distributions.
The lower bound on $U(f)$ follows from 
that $\text{sign-rank}(A) \geq \Omega(N^{1/4})$, and
the result of~\cite{DBLP:journals/jcss/PaturiS86}
that showed that unbounded error communication complexity
is equivalent to the logarithm of the sign rank.
See~\cite{DBLP:journals/cc/Sherstov10}
for more details.

\subsection{Counting VC classes}
Let $c(N,d)$ denote the number of classes $C\subseteq\{\pm 1\}^N$
with VC dimension $d$. 
We give the following estimate of $c(N,d)$
for constant $d$ and $N$ large enough.
The proof is given in Section~\ref{sec:counting}.

\begin{theorem}\label{thm:VCcount} \
%$$ N^{(1+o(1)){N \choose d}/(d+1)}
% \leq \alpha(N,d)\leq N^{O(d N^{d})}$$
%NA1
{For every $d > 0$, there is 
 $N_0=N_0(d)$ such that for all $N>N_0$:} 
$$ N^{(\Omega(N/d))^d}
\leq c(N,d)\leq N^{( O( N))^d}.$$
\end{theorem}

%\subsubsection{Counting maximum classes}
Let $m(N,d)$ denote the number of maximum classes $C\subseteq\{\pm 1\}^N$ of VC dimension $d$. The problem of estimating $m(N,d)$ was proposed by~\cite{FranklOpen}.
We provide the following estimate
(see Section~\ref{sec:counting}).

\begin{theorem}\label{thm:MAXcount} \
%NA1
{For every $d > 1$, there is $N_0=N_0(d)$ such that for all $N>N_0$:}
$$ N^{(1+o(1))\frac{1}{d+1}{N \choose d}}
\leq m(N,d)\leq N^{(1+o(1))\sum_{i=1}^d{N \choose i}}.$$
\end{theorem}

The gap between our upper and lower bound is roughly
a multiplicative factor of $d+1$ in the exponent.
In the previous bounds given by~\cite{FranklOpen} the gap was a multiplicative factor of $N$
in the exponent. 

\subsection{Counting graphs}

Here we describe an application of 
our method for proving Theorem~\ref{t41}
to counting graphs with a given forbidden substructure.

Let $G = (V,E)$ be a graph (not necessarily bipartite).
The universal graph $U(d)$ is defined as the bipartite graph with
two color classes $A$ and $B=2^A$ where
$|A| = d$, and the edges are defined as $\{a,b\}$ iff $a \in b$.
The graph $G$ is called $U(d)$-free if for all two disjoint sets of vertices
$A,B \subset V$ so that $|A|=d$ and $|B|=2^{d}$, the bipartite
graph consisting of all edges of $G$ between $A$ and $B$ 
is not isomorphic to $U(d)$.
In Theorem 24 of \cite{ABBM}, which improves Theorem 2 there,
it is proved that for $d \geq 2$,
the number of $U(d+1)$-free graphs on $N$ vertices is at most 
$$
2^{O(N^{2-1/d} (\log N)^{d+2})}.
$$
The proof in \cite{ABBM} 
is quite involved, consisting of several technical 
and complicated
steps. Our methods give a different, 
quick proof of an improved estimate,
replacing the $(\log N)^{d+2}$ term by a single $\log N$ term.
\begin{theorem}
\label{thm:Udfree}
For every fixed $d \geq 1$,
the number of $U(d+1)$-free graphs on $N$ vertices is at most
$2^{O(N^{2-1/d} \log N)}$.
\end{theorem}
The proof of the theorem is given in Section~\ref{sec:countingGraphs}.

\subsection{Geometry}
Differences and similarities 
between finite geometries and real geometry are well known.
An example of a related problem is 
finding the minimum dimension of Euclidean space in which we can embed
a given finite plane (i.e.\ a collection of points and lines
satisfying certain axioms).
By embed we mean that there are two one-to-one maps
$e_P,e_L$ so that $e_P(p) \in e_L(\ell)$ iff $p \in \ell$
for all $p \in P, \ell \in L$.
The Sylvester-Gallai theorem shows, for example,
that Fano's plane
cannot be embedded in any finite dimensional real space 
if points are mapped to points and lines to lines.

How about a less restrictive meaning of embedding?
One option is to allow embedding using half spaces,
that is, an embedding in which points are mapped to points
but lines are mapped to half spaces.
Such embedding is always possible if the dimension is high enough:
Every plane with point set $P$ and line set $L$ can be embedded in $\R^P$
by choosing $e_P(p)$ as the $p$'th unit vector,
and $e_L(\ell)$ as the half space with positive projection on the vector
with $1$ on points in $\ell$ and $-1$ on points outside $\ell$.
The minimum dimension for which such an embedding exists is captured by
the sign rank of the underlying incidence matrix
(up to a $\pm 1$).

\begin{corollary}
A finite projective plane of order $n \geq 3$  
cannot be embedded in $\R^k$ using half spaces,
unless $k > N^{1/4}-1$ with $N = n^2 + n+1$.
\end{corollary}

Roughly speaking, the corollary says that
there are no efficient ways to embed finite planes
in real space using half spaces.

\section{Proofs}
\subsection{Duality}
\label{sec:duality}
Here we discuss the connection between VC dimension and dual sign rank.

We start with an equivalent definition of dual sign rank,
that is based on the following notion.
We say that a set of columns $C$ is {\em antipodally shattered} in
a sign matrix $S$ if for each $v\in\{\pm 1\}^C$, either
$v$ or $-v$ appear as a row in the restriction of $S$ to the columns in $C$.
\begin{claim}
The set of columns $C$ is antipodally shattered in $S$ if and only if
in every matrix $M$ with $\text{sign}(M)=S$ the columns in $C$ are 
linearly independent.
\end{claim}
\begin{proof}
First, assume $C$ is such that there exists some
$M$ with $\text{sign}(M)=S$ in which the columns in $C$ are linearly dependent.
For a column $j \in C$, denote by $M(j)$ the $j$'th column in $M$.
Let $\{\alpha_j : j \in C\}$ be a set of real numbers so that
$\sum_{j \in C}\alpha_j M(j) =0$ and not all $\alpha_j$'s are zero. 
Consider the vector $v\in\{\pm 1\}^C$
such that $v_j = 1$ if $\alpha_j\geq 0$ and $v_j= -1$ if $\alpha_j < 0$.
The restriction of $S$ to $C$ does not contain $v$ nor $-v$
as a row, which certifies that $C$ is not antipodally shattered by $S$.

Second, let $C$ be a set of columns which is not antipodally shattered in $S$. 
Let $v\in\{\pm 1\}^C$ be such that both $v,-v$ do not appear as a row in the restriction
of $S$ to $C$. Consider the subspace $U =\{u \in \R^C : \sum_{j \in C} u_j v_j = 0 \}$.
%The space $V$ intersects every orthant in $\R^C$, except the orthant containing $v$ and the orthant containing $-v$.
For each sign vector $s \in\{\pm 1\}^C$ so that $s \neq \pm v$, 
the space $U$ contains some vector $u_s$ such that $\text{sign}(u_s)=s$. 
Let $M$ be so that $\text{sign}(M)=S$ 
and in addition for each row in $S$ that has pattern $s \in \{\pm\}^C$
in $S$ restricted to $C$, the corresponding row in $M$ restricted to $C$ is $u_s \in U$.
All rows in $M$ restricted to $C$ are in $U$, and therefore  
the set $\{M(j) : j \in C\}$ is linearly dependent.
\end{proof}

\begin{corollary}
The dual sign rank of $S$
is the maximum size of a set of columns that are antipodally shattered in $S$. 
\end{corollary}

Now, we prove Proposition~\ref{prop:dualsignrank}:
$$ VC(S)\leq \text{dual-sign-rank}(S) \leq 2VC(S)+1.$$

The left inequality:
The VC dimension of $S$
is at most the maximum size of a set of columns
that is antipodally shattered in $S$,
which by the above claim equals the dual sign rank of $S$.
%
%Let $j_1,\ldots,j_k$ be a set of columns in $S$ that are shattered in $S$ with $k = VC(S)$.
%Assume w.l.o.g.\ that $\{j_1,\ldots,j_k\} = [k]$.
%Let $M$ be a sign matrix so that $\text{sign}(M) = S$.
%Denote by $M(j)$ the $j$'th column of $M$.
%Assume towards a contradiction that $\sum_{j=1}^k \alpha_j M(j) = 0$, 
%where at least one of $\alpha_1,\ldots,\alpha_k$ are non zero.
%Let $J \subset [k]$ be the set of $j$'s so that $\alpha_j > 0$.
%Since $[k]$ is shattered in $S$,
%let $i$ be the row of $S$ so that
%for every $j \in J$ we have $S_{i,j} = 1$ and 
%for every $j \in [k] - J$ we have $S_{i,j} = -1$.
%Thus, the $i$'th entry in $\sum_{j=1}^k \alpha_j M(j) = 0$
%is non zero (a contradiction).

The right inequality:
Let $C$ be a largest set of columns that is antipodally shattered in $S$.
By the claim above, the dual sign rank of $S$ is $|C|$.
Let $A\subseteq C$ such that $|A| = \lfloor |C|/2 \rfloor$. If $A$ is
shattered in $S$ then we are done. Otherwise, there exists some $v\in\{\pm 1\}^A$
that does not appear in $S$ restricted to $A$. Since $C$ is antipodally
shattered by $S$, this implies that $S$ contains all patterns in $\{\pm 1\}^C$
whose restriction to $A$ is $-v$. In particular, $S$ shatters
$C\setminus A$ which is of size at least $\lfloor |C|/2\rfloor$.

\subsection{Sign rank versus VC dimension}
\label{sec:VCdim}
In this section we study the maximum possible sign rank of
$N \times N$ matrices with VC dimension $d$, presenting the proofs
of Proposition~\ref{prop:dualsignrank} and Theorems~\ref{t41} and~\ref{t42}. We also show that the arguments supply a new, short
proof and an improved estimate for a problem in
asymptotic enumeration of graphs studied by \cite{ABBM}. 

\subsubsection{VC dimension one}
\label{sec:VCone}
Our goal in this section is to show that sign matrices with VC dimension
one have sign rank at most $3$, and that $3$ is tight.
Before reading this section,
it may be a nice exercise to 
prove that the sign rank of the $N \times N$
signed identity matrix is exactly three (for $N \geq 4$).

Let us start by recalling a geometric interpretation of sign rank.
Let $M$ by an $R \times C$ sign matrix.
A $d$-dimensional embedding of $M$ using half spaces
consists of two maps $e_R,e_C$ so that
for every row $r \in [R]$ and column $c \in [C]$,
we have that $e_R(r) \in \R^d$,
$e_C(c)$ is a half space in $\R^d$,
and $M_{r,c} = 1$ iff $e_R(r) \in e_C(c)$.
The important property for us is that if $M$ has a $d$-dimensional
embedding using half spaces then its sign rank is at most $d+1$.
The $+1$ comes from the fact that the hyperplanes defining the half spaces
do not necessarily pass through the origin.

Our goal in this section is to embed $M$ with VC dimension
one in the plane using half spaces.
The embedding is constructive and uses the following 
known claim
(see, e.g., Theorem 11 in~\cite{DBLP:conf/alt/DoliwaFSZ15}).

\begin{claim}[\cite{DBLP:conf/alt/DoliwaFSZ15}]
\label{clm:ID}
Let $M$ be an $R \times C$ sign matrix with VC dimension one
so that no row appears twice in it, and every column $c$ is shattered
(i.e.\ the two values $\pm 1$ appear in it).
Then, there is a column $c_0 \in [C]$ and a row $r_0 \in [R]$
so that $M_{r_0,c_0} \neq M_{r,c_0}$ for all $r \neq r_0$ in $[R]$.
\end{claim}

\begin{proof}
For every column $c$, denote by
$\text{\it ones}_c$ the number of rows $r \in [R]$ so that $M_{r,c} = 1$,
and let $m_c = \min\{\text{\it ones}_c , R-\text{\it ones}_c\}$.
Assume without loss of generality that
$m_1 \leq m_c$ for all $c$, and that $m_1 = \text{\it ones}_1$.
Since all columns are shattered, $m_1 \geq 1$.
To prove the claim, it suffices to show that $m_1 \leq 1$.

Assume towards a contradiction that $m_1 \geq 2$.
For $b \in \{ 1,-1\}$,
denote by $M^{(b)}$ the submatrix of $M$ consisting of all rows
$r$ so that $M_{r,1} = b$.
The matrix $M^{(1)}$ has at least two rows.
Since all rows are different,
there is a column $c \neq 1$ so that two rows in $M^{(1)}$ differ in $c$.
Specifically, column $c$ is shattered in $M^{(1)}$.
Since $\text{VCdim}(M) = 1$,
it follows that $c$ is not shattered in $M^{(-1)}$,
which means that the value in column $c$
is the same for all rows of the matrix $M^{(-1)}$.
Therefore, $m_{c} < m_{1}$, which is a contradiction.
\end{proof}

The embedding we construct has an extra structure
which allows the induction to go through:
The rows are mapped to points on the unit circle
(i.e.\ set of points $x\in \R^2$ so that $\|x\|=1$).

\begin{lemma}
\label{lem:VCone}
Let $M$ be an $R \times C$ sign matrix of VC dimension one
so that no row appears twice in it.
Then, $M$ can be embedded in $\R^2$ using half spaces,
where each row is mapped to a point on
the unit circle.
\end{lemma}

The lemma immediately implies
Threorem~\ref{thm:VC1} due to the connection
to sign rank discussed above.

\begin{proof}
The proof follows by induction on $C$.
If $C=1$, the claim trivially holds.

The inductive step:
If there is a column that is not shattered,
then we can remove it, apply induction,
and then add a half space
that either contains or does not contain all points,
as necessary.
So, we can assume all columns are shattered. 
By Claim~\ref{clm:ID},
we can assume without loss of generality that
$M_{1,1} =1$ but $M_{r,1} = -1$ for all $r \neq 1$. 

Denote by $r_0$ the row of $M$ so that
$M_{r_0,c} = M_{1,c}$ for all $c \neq 1$, if such a row exists.
Let $M'$ be the matrix obtained from $M$ by deleting the first column,
and row $r_0$ if it exists, so that no row in $M'$ appears twice.
By induction, there is an appropriate embedding of $M'$ in $\R^2$.

The following is illustrated in Figure~1.
Let $x \in \R^2$ be the point on the unit circle to which the first row in $M'$
was mapped to
(this row corresponds to the first row of $M$ as well).
The half spaces in the embedding of $M'$ are defined by lines,
which mark the borders of the half spaces.
The unit circle intersects these lines in finitely many points.
Let $y,z$ be the two closest points to $x$ among
all these intersection points.
Let $y'$ be the point on the circle in the middle between $x,y$,
and let $z'$ be the point on the circle in the middle between $x,z$.
Add to the configuration one more half space 
which is defined by the line passing through $y',z'$.
If in addition row $r_0$ exists,
then map $r_0$ to the point $x_0$ on the circle which is right in the middle between $y,y'$.

%%%%%%%%%%%%%%GRAPHICS%%%%%%%%%%%%%%%%%%%
\begin{figure}[h]
\centering
\begin{tikzpicture}[scale=1]
\begin{scope}[xshift=-9cm,yshift = 1cm]
  % define a circle (center=O and \radius)
  \coordinate (O) at (1,2);
  \def\radius{4cm}
  % draw this circle and its center
  % define a random point (A) on this circle
  \path (O) ++(0.75*180:\radius) coordinate (Y);
  \path (O) ++(0.6*180:\radius) coordinate (YY);
  \path (O) ++(0.45*180:\radius) coordinate (X);
  \path (O) ++(0.675*180:\radius) coordinate (XX);
  \path (O) ++(0.4*180:\radius) coordinate (ZZ);
  \path (O) ++(0.35*180:\radius) coordinate (Z);
  \path (ZZ) ++(5,0) coordinate (ZZZ);
  \path (YY) ++(-5,0) coordinate (YYY);
  \path (ZZZ)++(0,2) coordinate (ZZU);
  \fill [black!10] (YYY) rectangle (ZZU);
  %\draw (O) circle(\radius);
  \draw[thick] ([shift=(20:\radius)] O) arc (20:160:\radius);
 % \fill (O) circle[radius=2pt] node[below left] {};
  \fill (Y) circle[radius=2pt] ++(0.8*180:1em) node {$y$};
  \fill (YY) circle[radius=2pt] ++(0.65*180:1em) node {$y'$};
  \fill (X) circle[radius=2pt] ++(0.5*180:1em) node {$x$};
  \fill (ZZ) circle[radius=2pt] ++(0.35*180:1em) node {$z'$};
  \fill (Z) circle[radius=2pt] ++(0.2*180:1em) node {$z$};
  \fill (XX) circle[radius=2pt] ++(0.9*180:1em) node {$x_0$};
\draw [thick] (YYY) -- (ZZZ);
%
% \path (YYY)++(0,0) coordinate (L11);
% \path (YYY)++(2,2) coordinate (L12);
% \path (L11)++(1,0) coordinate (L21);
% \path (L12)++(1,0) coordinate (L22);
%\path (L21)++(1,0) coordinate (L31);
% \path (L22)++(1,0) coordinate (L32);
% \path (L31)++(1,0) coordinate (L41);
% \path (L32)++(1,0) coordinate (L42);
%\path (L41)++(1,0) coordinate (L51);
% \path (L42)++(1,0) coordinate (L52);
%\path (L51)++(1,0) coordinate (L61);
% \path (L52)++(1,0) coordinate (L62);
%
%\draw [thin] (L11) -- (L12);
%\draw [thin] (L21) -- (L22);
%\draw [thin] (L31) -- (L32);
%\draw [thin] (L41) -- (L42);
%\draw [thin] (L51) -- (L52);
%\draw [thin] (L61) -- (L62);
\end{scope}
\end{tikzpicture}
\caption{An example of a neighbourhood of $x$.
All other points in embedding of $M'$ are to left of $y$
and right of $z$ on the circle.
The half space defined by the line through $y',z'$ is coloured light gray.}
\end{figure}
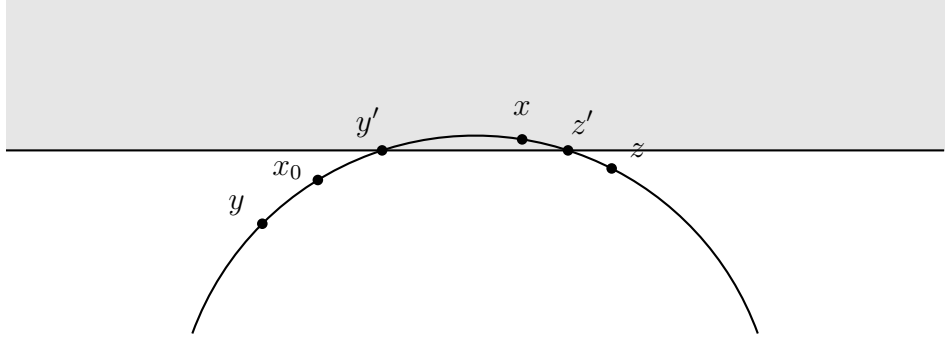

This is the construction.  
Its correctness follows by induction, 
by the choice of the last added half space
which separates $x$ from all other points,
and since if $x_0$ exists it 
belongs to the same cell as $x$ in the embedding of $M'$.
%%%%%%%%%%%%%%%%%%%ENDGRAPHICS%%%%%%%%%%%%%%
\end{proof} 

We conclude the section
by showing that the bound $3$ above cannot be improved.

\begin{proof}[Proof of Claim~\ref{clm:Id4}]
One may deduce the claim from Forster's argument,
but we provide a more elementary argument.
It suffices to consider the case $N=4$.
Consider an arrangement of four half planes in $\R^2$.
These four half planes partition $\R^2$ to eight cones
with different sign signatures,
as illustrated in Figure~2.
Let $M$ be the $8\times 4$ sign matrix 
whose rows are these sign signatures.
The rows of $M$ form a distance preserving 
cycle (i.e.\ the distance along cycle is hamming distance) of length eight in the discrete cube 
of dimension four\footnote{The graph with vertex set $\{\pm 1\}^4$
where every two vectors of hamming distance one are connected by an edge.}.

Finally, the signed identity matrix
is not a submatrix of $M$.
To see this, note that the four rows of the signed identity matrix have
pairwise hamming distance two,
but there are no such four points (not even three points) on this cycle
of length eight. 

%%%%%%%%%%%%%%%%%%GRAPHICS%%%%%%%%%%%%%%%%%
\begin{figure}[h]
\centering
\begin{tikzpicture}[scale=1]
\begin{scope}[xshift=-9cm,yshift = 10cm]
 \coordinate (O) at (1,2);
  \def\radius{2.5cm}
  % define a random point (A) on this circle
  \path (O) ++(0.0*180:\radius) coordinate (A);
  \path (O) ++(1.0*180:\radius) coordinate (AA);
  \path (O) ++(0.1*180:\radius) coordinate (B);
  \path (O) ++(1.1*180:\radius) coordinate (BB);
  \path (O) ++(0.4*180:\radius) coordinate (C);
  \path (O) ++(1.4*180:\radius) coordinate (CC);
  \path (O) ++(0.6*180:\radius) coordinate (D);
  \path (O) ++(1.6*180:\radius) coordinate (DD);
  \path (O) ++(0.0*180:1.1*\radius) coordinate (A1);
  \path (O) ++(1.0*180:1.1*\radius) coordinate (AA1);
  \path (O) ++(0.1*180:1.1*\radius) coordinate (B1);
  \path (O) ++(1.1*180:1.1*\radius) coordinate (BB1);
  \path (O) ++(0.5*180:1.1*\radius) coordinate (C1);
  \path (O) ++(1.5*180:1.1*\radius) coordinate (CC1);
  \path (O) ++(0.7*180:1.1*\radius) coordinate (D1);
  \path (O) ++(1.7*180:1.1*\radius) coordinate (DD1);
  \path (O) ++(0.05*180:\radius) coordinate (A2);
  \path (O) ++(1.05*180:\radius) coordinate (AA2);
  \path (O) ++(0.3*180:\radius) coordinate (B2);
  \path (O) ++(1.3*180:\radius) coordinate (BB2);
  \path (O) ++(0.6*180:\radius) coordinate (C2);
  \path (O) ++(1.6*180:\radius) coordinate (CC2);
  \path (O) ++(0.85*180:\radius) coordinate (D2);
  \path (O) ++(1.85*180:\radius) coordinate (DD2);
  \draw  (A1) -- (AA1);
  \draw  (B1) -- (BB1);
  \draw  (C1) -- (CC1);
  \draw  (D1) -- (DD1);
 %\draw (O) circle[radius=\radius];
  \fill (O) circle[radius=2pt] node[below left] {};
\fill (A2) circle[radius=2pt] (A2)++(0.9,0) node {$++++$};
\fill (B2) circle[radius=2pt] (B2)++(0.9,0) node {$+++-$};
\fill (C2) circle[radius=2pt] (C2)++(0,0.4) node {$++--$};
\fill (D2) circle[radius=2pt] (D2)++(-0.9,0) node {$+---$};
\fill (AA2) circle[radius=2pt] (AA2)++(-0.9,0) node {$----$};
\fill (BB2) circle[radius=2pt] (BB2)++(-0.9,0) node {$---+$};
\fill (CC2) circle[radius=2pt] (CC2)++(0,-0.4) node {$--++$};
\fill (DD2) circle[radius=2pt] (DD2)++(0.9,0) node {$-+++$};
\end{scope}
\end{tikzpicture}
\caption{Four lines defining four half planes,
and the corresponding eight sign signatures.}
\end{figure}
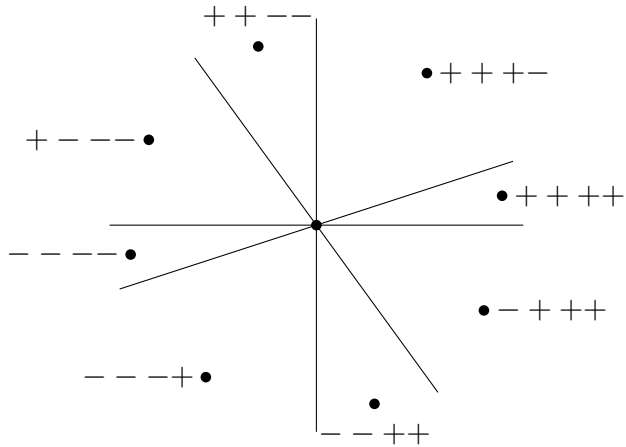
%%%%%%%%%%%%%%%%%%ENDGRAPHICS%%%%%%%%%%%%%%%
\end{proof}

\subsubsection{The upper bound}
In this subsection we prove Theorem \ref{t41}. The proof is short,
but requires several ingredients.
The first one has been mentioned already, and appears in 
\cite{DBLP:conf/focs/AlonFR85}.  For a sign matrix $S$, let
$SC(S)$ denote the maximum number of sign changes (SC) along a column of
$S$. Define $SC^*(S)=\min SC(M)$ where the minimum is taken over all
matrices $M$ obtained from $S$ by a permutation of the rows.
\begin{lemma}[\cite{DBLP:conf/focs/AlonFR85}]
\label{l43}
For any sign matrix $S$,
$\text{sign-rank}(S) \leq SC^*(S)+1$.
\end{lemma}
Of course we can replace here rows by columns, but for our purpose 
the above version will do.
The second result we need is a theorem of~\cite{We} (see also
\cite{CW}). As observed, for example, in \cite{MWW}, plugging
in its proof a result of~\cite{Ha} improves it by a
logarithmic factor, yielding the result we describe next.
For a function $g$ mapping positive integers to positive integers,
we say that a sign matrix $S$ satisfies a primal shatter function $g$
if for any integer $t$ and any set $I$ of $m$ columns of $S$, the number of
distinct projections of the rows of $S$ on $I$ is at most $g(t)$. 
%A well known lemma proved in several papers including
%\cite{Sa} asserts that if the VC dimension of $S$ is $d$, then
%it satisfies the primal shatter function $\sum_{i=0}^d {t
%\choose i}$, and hence also the function $g(t)=t^d$.
The result of Welzl (after its optimization following
\cite{Ha}) can be stated as follows\footnote{The statement in
\cite{We} and the subsequent papers is formulated in terms of 
somewhat different
notions, but it is not difficult to check that it is equivalent to
the statement below.}.

\begin{lemma}[\cite{We}, see also \cite{CW, MWW}]
\label{l44}
Let $S$ be a sign matrix with $N$ rows that satisfies the primal
shatter function $g(t)=c  t^d$ for some constants 
$c \geq 0$ and $d >1$. Then $SC^*(S) \leq O(N^{1-1/d})$.
\end{lemma}

\begin{proof}[Proof of Theorem \ref{t41}]
Let $S$ be an $N \times N$ sign matrix of VC dimension $d>1$. By
Sauer's lemma \cite{Sa}, 
it satisfies the primal shatter function $g(t) = t^d$.
Hence, by Lemma \ref{l44}, $SC^*(S) \leq O(N^{1-1/d})$. Therefore, by
Lemma~\ref{l43}, $\text{sign-rank}(S) \leq O(N^{1-1/d})$.
\end{proof}

\paragraph{On the tightness of the argument.}
\label{sec:tightness}
The proof of Theorem \ref{t41} works, with essentially no change, 
for a larger  class of sign matrices than the ones with VC dimension
$d$. Indeed, the proof shows that the sign rank of any $N \times N$
matrix with primal shatter function at most $ct^d$ for some fixed $c$
and $d>1$ is at most $O(N^{1-1/d}).$ In this statement the estimate
is sharp for all integers $d$, up to a logarithmic factor. This
follows from the construction in \cite{ARS}, which supplies 
$N \times N$ boolean matrices so that the number of $1$ entries
in them is at least $\Omega(N^{2-1/d})$, and
they contain no $d$ by $D=(d-1)!+1$ submatrices of $1$'s. 
These matrices satisfy the primal shatter function
$g(t)=D {t \choose d}
+\sum_{i=0}^{d-1} {t \choose i} $ (with room to spare). Indeed, if
we have more than that many  distinct projections  on a set of $t$
columns, we can omit all projections of weight at most $d-1$. Each
additional projection contains $1$'s in at least one set of size
$d$, and the same $d$-set cannot %must 
be covered  more than $D$ times.
Plugging this matrix in the counting argument that gives a lower
bound for the  sign rank using Lemma~\ref{l46} 
proven below supplies
an $\Omega(N^{1-1/d}/\log N)$ lower bound for the sign rank of
many $N \times N$ matrices with primal shatter function
$O(t^d)$.

We have seen in Lemma~\ref{l43}
that sign rank is at most of order $SC^*$. 
Moreover, for a fixed $r$, many of the $N\times N$ sign
matrices with sign rank at most $r$
also have $SC^*$ at most $r$: Indeed,
a simple counting argument 
shows that the number of $N\times N$ sign matrices $M$
with $SC(M) < r$ is
$$\left(2\cdot\sum_{i=0}^{r-1}{N-1 \choose i}\right)^N=2^{\Omega(rN\log N)},$$
so, the set of $N\times N$ sign matrices with $SC^*(M) < r$ 
is a subset of size $2^{\Omega(rN\log N)}$ 
of all $N\times N$ sign matrices with sign rank at most $r$.

How many $N\times N$ matrices of sign rank at most $r$
are there? by Lemma~\ref{l46} proved in the next section,
this number is at most $2^{O(rN \log N )}$. 
So, the set of matrices with $SC^*<r$ is a rather large
subset of the set of matrices with sign rank at most $r$.

It is reasonable, therefore, to wonder whether an inequality in the other direction holds.
%That is, if all low sign rank matrices come from
%the construction using univariate polynomials as discussed after
%Theorem~\ref{thm:SpecVsSignRank} and in \cite{DBLP:conf/focs/AlonFR85}.
Namely, whether all matrices of sign rank $r$ have $SC^*$ order of $r$.
We now describe an example which shows that this is far
from being true, and also 
demonstrates the tightness of Lemma~\ref{l44}.
Namely, for every constant $d>1$, there are $N \times N$ matrices $S$,
which satisfy the primal shatter function $g(t) = c t^d$ for a constant $c$,
and on the other hand $SC^*(S) \geq \Omega(N^{1-1/d})$.
Consider the grid of points $P = [n]^d$ as a subset of $\R^d$.
Denote by $e_1,\ldots,e_d$ the standard unit vectors in $\R^d$.
For $i \in [n-1]$ and $j \in [d]$, define the hyperplane
$h_{i,j} = \{x : \langle x , e_j \rangle > i + (1/2)\}$.
Denote by $H$ the set of these $d(n-1)$ axis parallel hyperplanes.
Let $S$ be the $P \times H$ sign matrix defined by $P$ and $H$.
That is, $S_{p,h} = 1$ iff $p \in h$.
First, the matrix $S$ satisfies the primal shatter
function $c t^d$, since every family of $t$ hyperplanes
partition $\R^d$ to at most $c t^d$ cells.
Second, we show that 
$$SC^*(S) \geq \frac{n^{d}-1}{d(n-1)} \geq \frac{|P|^{1-1/d}}{d}.$$
Indeed, fix some order on the rows of $S$,
that is, order the points $P = \{p_1,\ldots,p_N\}$ with $N = |P|$.
The key point is that one of the hyperplanes $h_0 \in H$
is so that the number of $i \in [N-1]$ for which
$S_{p_i,h_0} \neq S_{p_{i+1},h_0}$ is at least $(n^{d}-1)/(d(n-1))$:
For each $i$ there is at least one hyperplane $h$
that separates $p_i$ and $p_{i+1}$, that is, for which $S_{p_{i},h} \neq S_{p_{i+1},h}$.
The number of such pairs of points is $n^d-1$,
and the number of hyperplanes is just $d(n-1)$.

\subsubsection{The lower bound}
In this subsection we prove Theorem \ref{t42}. Our approach follows
the one of \cite{DBLP:conf/focs/AlonFR85}, which is based on
known bounds for the number of sign patterns of real polynomials.
A similar approach has been subsequently used by 
\cite{DBLP:journals/jmlr/Ben-DavidES02} to derive lower bounds
for $f(N,d)$ for $d \geq 4$, but here we do it in a slightly more
sophisticated way and get better bounds. 

Although we can 
use the estimate in \cite{DBLP:conf/focs/AlonFR85} for the
number of sign matrices with a given sign rank, we prefer to
describe the argument
by directly applying a result of~\cite{Wa}, described next.

Let $P=(P_1, P_2, \ldots ,P_m)$ be a list of $m$
real polynomials, each in $\ell$ variables.
Define the semi-variety
$$V = V(P) =\{x \in \R^\ell :P_i(x )\neq 0~ \mbox{for all}~  1\le i\le m\}.$$
For $x \in V$, the sign pattern of $P$ at $x$ is the vector
$$(sign(P_1(x)), sign(P_2(x)), \ldots ,sign(P_m(x))) \in \{-1,1\}^m.$$
Let $s(P)$ be the total number
of sign patterns of $P$ as $x$ ranges over all of $V$.
This number is bounded from above by the number of connected components of $V$.

\begin{theorem}[\cite{Wa}]
\label{t45}
Let $P=(P_1, P_2, \ldots ,P_m)$ be a list of 
real polynomials, each in $\ell$ variables and of degree at most $k$.
If $m\geq\ell$
then the number of  connected components of
$V(P)$ 
(and hence also $s(P)$)
%=\{x \in R^\ell :P_i(x )\neq 0 ~ \mbox{ for all }~  1\le i\le
%m\}$
is at most  $(4 ekm/\ell )^\ell$. 
%Specifically, $s(P) \leq (4 ekm/\ell )^\ell$.
\end{theorem}

An $N \times N$ matrix $M$ is of rank at most $r$ iff it can be written as
a product $M=M_1 \cdot M_2$ of an $N \times r$ matrix $M_1$ by
an $r \times N$ matrix $M_2$. Therefore, each entry of $M$ is a quadratic
polynomial in the $2Nr$ variables describing the entries of
$M_1$ and $M_2$. We thus deduce the following from Warren's Theorem
stated above. A similar argument has been used by~\cite{DBLP:journals/ml/Ben-DavidL98}.

\begin{lemma}
\label{l46}
Let $r\leq N/2$. Then,
the number of $N \times N$ sign matrices of sign rank at most
$r$ does not exceed
$(O(N/r))^{2Nr}  \leq 2^{O(rN \log N )}$.  
\end{lemma}

For a fixed $r$, this bound
for the logarithm of the above quantity
is tight up to a constant factor: 
As argued in Subsection~\ref{sec:tightness}, there are at least
some $2^{\Omega(rN\log N)}$ matrices of sign rank $r$.

In order to derive the statement of Theorem \ref{t42} from the last
lemma it suffices to show that the number of $N \times N$
sign matrices of VC dimension $d$ is sufficiently 
large. We proceed to do so. 
It is more convenient to discuss boolean matrices in what follows
(instead of their signed versions).
%,
%the transformation from those to sign matrices is obvious.
\begin{proof}[Proof of Theorem~\ref{t42}]
There are $4$ parts as follows.

\noindent
1. The case $d=2$: 
Consider the $N \times N$ incidence matrix $A$ of the projective plane with
$N$ points and $N$ lines, considered in the previous sections.
The number of $1$ entries in $A$ is $(1+o(1)) N^{3/2}$,
and it does not contain $J_{2 \times 2}$ 
(the $2 \times 2$ all $1$ matrix) as a submatrix, 
since there is only one line  passing through any two given points. 
Therefore, any matrix obtained  from
it by replacing ones by zeros has VC dimension at most $2$, since 
every matrix of VC dimension $3$ must contain $J_{2 \times 2}$
as a submatrix.
%in order
%to have a shattered set of three columns we need at least two distinct  
%rows that contain $1$ in any pair of these columns. 
This gives us
$2^{(1+o(1))N^{3/2}}$ distinct $N \times N$ sign matrices of
VC dimension at most $2$. Lemma~\ref{l46} therefore establishes the
assertion of Theorem~\ref{t42}, part 1.
\vspace{0.1cm}

\noindent
2. The case $d=3$:
Call a $5 \times 4$ binary matrix heavy if its rows are the all
$1$ row and the $4$ rows with Hamming weight $3$. Call a $5 \times 4$
boolean matrix heavy-dominating if there is a heavy matrix which 
is smaller or equal to it in every entry.

We claim that there is a boolean $N \times N$ matrix
$B$ so that the number of $1$ entries in it
is at least $\Omega(N^{23/15})$, and it does not
contain any heavy-dominating $5 \times 4$ submatrix. 
Given such a matrix $B$,
any matrix obtained
from $B$ by replacing some of the ones by zeros
have VC dimension at most $3$.
This %will give us a large number of matrices of
%VC dimension at most $3$ and will 
implies part 2 of Theorem~\ref{t42}, using
Lemma \ref{l46} as before.

The existence of $B$ is proved by a
probabilistic argument. Let $C$ be a random binary matrix  in which
each entry, randomly and  independently, is $1$ with probability
$p=\frac{1}{2N^{7/15}}$. Let $X$ be the random variable counting the
number of $1$ entries of $C$ minus twice the
number of $5 \times 4$ heavy-dominant submatrices $C$ contains. 
By linearity of expectation, 
%the expectation  of $X$  satisfies
$$
\E(X) \geq N^2p-2 N^{4+5}p^{1 \cdot 4+4 \cdot 3} %=N^2p-2 N^9p^{16}
=\Omega(N^{23/15}).
$$
Fix a matrix $C$ for which the value of $X$ is at least its
expectation. Replace at most two $1$ entries by $0$
in each heavy-dominant $5 \times 4$ submatrix in $C$
to get the required matrix $B$.
\vspace{0.1cm}

\noindent
3. The case $d=4$:
The basic idea is as before, but here there is an explicit
construction that beats the probabilistic one. Indeed,~\cite{Br} constructed an $N \times N$ boolean matrix $B$ 
so that the number of $1$ entries in $B$ is at least
$\Omega(N^{5/3})$ %$1$ entries 
and it does not contain $J_{3 \times 3}$ as a submatrix
%no $3$ by $3$ matrix of $1$s
(see also \cite{ARS} for another construction). %It is easy to check
%that 
No set of $5$ rows in every matrix obtained from this one by
replacing $1$'s by $0$'s can be shattered, implying the desired
result as before.
\vspace{0.1cm}

\noindent
4. The case $d>4$:
The proof here is similar to the one in part 2. We prove by a
probabilistic argument that there is an $N \times N$ binary matrix $B$ 
so that the number of $1$ entries in it is at least 
$$\Omega(N^{2-(d^2+5d+2)/(d^3+2d^2+3d)})$$
and it contains no heavy-dominant submatrix.
Here, heavy-dominant means a $1+(d+1)+{{d+1} \choose 2}$ by
$d+1$ matrix that is bigger or equal in each entry
than the matrix whose rows are all the distinct vectors of length $d+1$
and Hamming weight at least $d-1$. 
Any matrix obtained by replacing $1$'s by $0$'s in $B$
cannot have VC dimension exceeding $d$. The result follows,
again, from Lemma \ref{l46}. 

We start as before with a random matrix $C$ in
which each entry, randomly and independently, is chosen to be $1$
with probability 
\begin{align*}
p 
& = \frac{1}{2} \cdot 
N^{\frac{2-1-(d+1)-{{d+1} \choose 2}-(d+1)}{1 \cdot (d+1)+
(d+1) \cdot d +{{d+1} \choose 2} \cdot (d-1) - 1}} 
%\\
%& = \frac{1}{2} \cdot 
%N^{- \frac{2d +1 + ((d+1)d/2) }{d+1+
%d^2+d + ((d+1)d(d-1)/2) - 1}} \\
%& 
%= \frac{1}{2} \cdot 
%N^{- \frac{1}{d} \frac{5d+2 +  d^2   }{3+ 2d+ d^2}} .
= \frac{1}{2N^{(d^2+5d+2)/(d^3+2d^2+3d)}} .
\end{align*}
Let $X$ be the random variable counting the
number of $1$ entries of $C$ minus three times the
number of heavy-dominant submatrices $C$ contains. 
As before, $\E(X) \geq \Omega(N^2 p)$,
and by deleting some of the $1$'s in $C$ we get $B$.
\end{proof}

\subsection{Sign rank and spectral gaps}
\label{sec:LBproof}

The lower bound on the sign rank uses Forster's argument 
\cite{DBLP:conf/coco/Forster01}, who showed how to 
relate sign rank to spectral norm.
He proved that if $S$ is an $N \times N$ sign matrix then
$$\text{sign-rank}(S) \geq  \frac{N}{\|S\|}.$$
We would like to apply Forster's theorem to the matrix $S$
in our explicit examples.
The spectral norm of $S$, however, is too large to be useful: 
If $S$ is  
$\Delta \leq N/3$ regular and $x$ is the all $1$ vector then
$Sx = (2\Delta-N) x$ and
so $\|S\| \geq N/3$.
Applying Forster's theorem to $S$ yields that its sign rank is $\Omega(1)$,
which is not informative.

Our solution is based on the observation that Forster's argument actually proves a stronger statement.
His proof works as long as the entries of the matrix are not too close to zero,
as was already noticed in \cite{DBLP:conf/fsttcs/ForsterKLMSS01}.
We therefore use a variant of the spectral norm of a sign matrix $S$ 
which we call star norm and denote by\footnote{
The minimizer belongs to a closed subset of the bounded 
set $\{M: \|M\|\leq \|S\| \}$.}
$$\sn{S} = \min \{ \| M \| : M_{i,j} S_{i,j} \geq 1
\text{ for all $i,j$} \}.$$
Three comments seem in place.
(i) We do not think of the star norm as a norm. 
(ii) It is always at most the spectral norm, $\sn{S} \leq \|S\|$.
(iii) Every $M$ in the above minimum
satisfies $\text{sign-rank}(M) = \text{sign-rank}(S)$.
\begin{theorem}[\cite{DBLP:conf/fsttcs/ForsterKLMSS01}]
\label{thm:Forster}
Let $S$ be an $N \times N$ sign matrix.
Then,
$$\text{sign-rank}(S) \geq  \frac{N}{\sn{S}}.$$
\end{theorem}

For completeness,
in Section~\ref{sec:ForsterProof} we provide a short proof
of this theorem
(which uses the main lemma from \cite{DBLP:conf/coco/Forster01}
as a black box).
To get any improvement using this theorem,
we must have $\sn{S} \ll \|S\|$.
It is not a priori obvious that there is a matrix $S$ for which
this holds.
The following lemma shows that spectral gaps
yield such examples.

\begin{theorem}
\label{thm:normV}
Let $S$ be a $\Delta$ regular $N \times N$ sign matrix
with $\Delta \leq N/2$,
and $B$ its boolean version.
Then,
$$\sn{S} \leq \frac{N \cdot \sigma_2(B)}{\Delta}.$$
\end{theorem}

In other words,
every regular sign matrix whose boolean version has a spectral gap
has a small star norm.
Theorem~\ref{thm:Forster} and Theorem~\ref{thm:normV}
immediately imply Theorem~\ref{thm:SpecVsSignRank}.
In Section~\ref{sec:IntroFinGeo},
we provided concrete examples of matrices with a spectral gap,
that have applications in communication complexity,
learning theory and geometry.

\begin{proof}[Proof of Theorem~\ref{thm:normV}]
Define the matrix
$$M = \frac{N}{\Delta} B - J.$$
Observe that since $N \geq 2 \Delta$ it follows that
$M_{i,j} S_{i,j} \geq 1$ for all $i,j$.
So, 
$$\sn{S} \leq \|M\|.$$
Since $B$ is regular, the all $1$ vector $y$
is a right singular vector of $B$ with singular value $\Delta$.
Specifically, $M y=0$.
For every $x$, write $x=x_1+x_2$ where $x_1$ is the projection
of $x$ on $y$ and $x_2$ is orthogonal to $y$.
Thus,
$$\langle M x , M x\rangle
= \langle M x_2 , M x_2 \rangle
= \frac{N^2}{\Delta^2} \langle B x_2 , B x_2 \rangle.$$
Note that
$\|B\| \leq \Delta$ (and hence $\| B \|=\Delta$). Indeed, since
$B$ is regular,
there are $\Delta$ permutation matrices $B^{(1)},\ldots,B^{(\Delta)}$
so that $B$ is their sum.
The spectral norm of each $B^{(i)}$ is one.
The desired bound follows by the triangle inequality.

Finally, since $x_2$ is orthogonal to $y$,
$$\|B x_2\| \leq \sigma_2(B) \cdot \|x_2\| \leq \sigma_2(B) \cdot \|x\|.$$
So,
$$\|M\| \leq \frac{N \cdot \sigma_2(B)}{\Delta}.$$
\end{proof}

\subsubsection{Limitations}
It is interesting to understand whether
the approach above can give a better lower bound on sign rank.
There are two parts to the argument: 
Forster's argument, and the upper bound on $\sn{S}$.
We can try to separately improve each of the two parts.

Any improvement over Forster's argument
would be very interesting,
but as mentioned there is no significant improvement
over it even without the restriction induced by VC dimension,
so we do not discuss it further.

To improve the second part,
we would like to find examples with the biggest spectral gap possible.
The Alon-Boppana theorem \cite{zbMATH00031582} 
optimally describes limitations on spectral gaps.
The second eigenvalue $\sigma$
of a $\Delta$ regular graph is not too small,
$$\sigma \geq 2 \sqrt{\Delta-1} - o(1),$$
where the $o(1)$ term vanishes when $N$ tends to infinity
(a similar statement holds when the diameter is 
large \cite{zbMATH00031582}).
Specifically, the best lower bound on sign rank 
this approach can yield is roughly $\sqrt{\Delta}/2$,
at least when $\Delta \leq N^{o(1)}$.

But what about general lower bounds on $\sn{S}$?
It is well known that any $N \times N$
sign matrix $S$ satisfies $\|S\| \geq \sqrt{N}$.
We prove a generalization of this statement.

\begin{lemma}
\label{lem:min*norm}
Let $S$ be an $N \times N$ sign matrix.
For $i \in [N]$, let $\gamma_i$ be the minimum between the number of $1$'s and the number of $-1$'s in the i'th row.
Let $\gamma = \gamma(S) = \max \{\gamma_i : i \in [N]\}$.
Then,
$$\sn{S} \geq \frac{  N-\gamma}{\sqrt{\gamma}+1}.$$
\end{lemma}
This lemma provides limitations on the bound from
Theorem~\ref{thm:normV}.
Indeed, %for all $S$, 
$\gamma(S)\leq\frac{N}{2}$
and $\frac{N-\gamma}{\sqrt{\gamma}+1}$ is a monotone
decreasing function of $\gamma$,
which implies 
$\sn{S} %\geq\frac{\frac{N}{2}}{\sqrt\frac{N}{2}+1}=
\geq \Omega(\sqrt{N})$.
Interestingly, Lemma~\ref{lem:min*norm} 
and Theorem~\ref{thm:normV} provide 
a quantitively weaker but a more general statement
than the Alon-Boppana theorem:
If $B$ is a $\Delta$ regular $N \times N$ 
boolean matrix with $\Delta \leq N/2$,
then
$$\frac{N \cdot \sigma_2(B) }{\Delta} %\geq \sn{S}
\geq \frac{N-\Delta}{\sqrt{\Delta}+1}
\ \ \Rightarrow \ \
\sigma_2(B) \geq
\left( 1 - \frac{\Delta}{N} \right) \left(\sqrt{\Delta}-1\right).$$
This bound is off by roughly a factor 
of two when the diameter of the graph is large.
When the diameter is small, like in the case of the projective plane
which we discuss in more detail below, this bound is actually
almost tight:
The second largest singular value of the boolean 
point-line incidence matrix
of a projective plane of order $n$ is 
$\sqrt{n}$ while this matrix is $n+1$ regular
(c.f., e.g.,~\cite{Al2}).

It is perhaps worth noting that in fact here there is a simple
argument that gives a slightly stronger result
for boolean regular matrices. The sum of
squares of the singular values of $B$ is the trace of 
$B^tB$, which is $N \Delta$. As the spectral norm is $\Delta$, the
sum of squares of the other singular values is $N \Delta -\Delta^2=
\Delta(N-\Delta)$, implying that 
$$
\sigma_2(B) \geq \sqrt {\frac{\Delta(N-\Delta)}{N-1}} ,$$
which is (slightly) larger than the bound above.

\begin{proof}[Proof of Lemma~\ref{lem:min*norm}]
Let $M$ be a matrix so that $\|M\|= \sn{S}$ and
$M_{i,j} S_{i,j} \geq 1$ for all $i,j$.
Assume without loss of generality\footnote{Multiplying a row by $-1$ does not affect $\sn{S}$.} that $\gamma_i$
is the number of $-1$'s in the $i$'th row of $S$.
If $\gamma =0$, then $S$ has only positive entries
which implies $\|M\| \geq N$ as claimed.
So, we may assume $\gamma \geq 1$.
Let $t$ be the largest real so that 
\begin{align}
\label{eqn:tis}
t^2 = \frac{(N - \gamma - t)^2}{\gamma}.
\end{align}
That is,
if $\gamma = 1$ then
$t = \frac{N-\gamma}{2}$
and if $\gamma > 1$ then
\begin{align*}
t &=
\frac{ -(N-\gamma) + \sqrt{(N-\gamma)^2 + (\gamma-1)(N-\gamma)^2}}{\gamma-1} .
\end{align*}
In both cases,
\begin{align*}
t = 
\frac{ N-\gamma}{\sqrt{\gamma}+1} .
\end{align*}

We shall prove that
$$\|M\| \geq t.$$
There are two cases to consider.
One is that for all $i \in [N]$ we have
$\sum_j M_{i,j} \geq t$.
In this case, if $x$ is the all $1$ vector then
$$\|M\| \geq \frac{\|M x\|}{\|x\|} \geq  t.$$
The second case is that
there is $i \in [N]$ so that $\sum_j M_{i,j} < t$.
Assume without loss of generality that $i=1$.
Denote by $C$ the subset of the columns $j$ so that $M_{1,j} < 0$.
Thus,
\begin{align*}
\sum_{j \in C} |M_{1,j}| &> \sum_{j \not \in C} M_{1,j} - t\\
			     &\geq |[N] \setminus C| - 
                               t \tag{$|M_{i,j}|\geq 1$ 
                               for all $i,j$}\\
		                &\geq N-\gamma-t. \tag{$|C|\leq
                                \gamma$}
\end{align*}
Convexity of $x \mapsto x^2$ implies that
$$\left(\sum_{j \in C} |M_{1,j}| \right)^2
\leq |C| \sum_{j \in C} {M}_{1,j}^2,$$
so by \eqref{eqn:tis}
$$\sum_j {M}_{1,j}^2 \geq \frac{(N - \gamma - t)^2}{\gamma} = t^2.$$
In this case, if $x$ is the vector with $1$ in the first entry
and $0$ in all other entries then
$$\|(M)^T x\| = \sqrt{\sum_j {M}_{1,j}^2}  \geq t = t\|x\|.$$
Since $\|(M)^T\|=\|M\|$, it follows that
$\|M\| \geq t$.
\end{proof}

\subsubsection{Forster's theorem}
\label{sec:ForsterProof}

Here we provide a proof of Forster's theorem,
that is based on the following key lemma,
which he proved.
\begin{lemma}[\cite{DBLP:conf/coco/Forster01}]
\label{lem:Forster}
Let $X \subset \R^k$ be a finite set in general position,
i.e., every $k$ vectors in it are linearly independent.
Then, there exists an invertible matrix $B$ so that
$$\sum_{x \in X} \frac{1}{\|B x\|^2}
Bx \otimes Bx = \frac{|X|}{k} I,$$
where $I$ is the identity matrix,
and $Bx \otimes Bx$ is the rank one matrix
with $(i,j)$ entry $(Bx)_i (Bx)_j$.
\end{lemma}

The lemma shows that every $X$ in general position
can be linearly mapped to $BX$ that is, in some sense, equidistributed.
In a nutshell, the proof of the lemma is by finding
$B_1,B_2,\ldots$ so that each $B_i$ makes $B_{i-1}X$ closer
to being equidistributed,
and finally using that the underlying object is compact,
so that this process reaches its goal.

\begin{proof}[Proof of Theorem~\ref{thm:Forster}]
Let $M$ be a matrix so that $\|M\| = \sn{S}$
and $M_{i,j} S_{i,j} \geq 1$ for all $i,j$.
Clearly, $\text{sign-rank}(S) = \text{sign-rank}(M)$.
Let $X ,Y$ be two subsets of size $N$ of unit vectors in
$\R^k$ with $k = \text{sign-rank}(M)$
so that $\langle x ,y\rangle M_{x,y} > 0$ for all $x,y$.
Lemma~\ref{lem:Forster} says that we can assume
\begin{align}
\label{eqn:xox}
\sum_{x \in X} x \otimes x = \frac{N}{k} I;
\end{align}
If necessary replace $X$ by $BX$ and $Y$ by $(B^T)^{-1}Y$,
and then normalize
(the assumption required in the lemma that
$X$ is in general position may be obtained by a slight perturbation of its vectors).

The proof continues by bounding 
$D =  \sum_{x \in X,y\in Y} M_{x,y} \langle x, y \rangle$
in two different ways.

\noindent
First, bound $D$ from above:
Observe that for every two vectors $u,v$,
Cauchy-Schwartz inequality implies
\begin{align}
\label{eqn:normUV}
\langle M u , v \rangle 
\leq \|M u\| \|v\|
\leq \|M\| \|u\| \|v\|.
\end{align}
Thus,
\begin{align*}
D 
& = 
\sum_{i=1}^k \sum_{x \in X} \sum_{y \in Y}  M_{x,y} x_i y_i\\  
& \leq 
 \sum_{i=1}^k \|M\| \sqrt{\sum_{x \in X} x_i^2}
\sqrt{\sum_{y \in Y} y_i^2} \tag{\eqref{eqn:normUV}} \\
& \leq 
 \|M\| \sqrt{\sum_{i=1}^k \sum_{x \in X} x_i^2}
\sqrt{\sum_{i=1}^k \sum_{y \in Y} y_i^2} \tag{Cauchy-Schwartz}  = \|M\| N.
\end{align*}

\noindent
Second, bound $D$ from below:
Since $|M_{x,y}|\geq 1$ and $|\langle x, y \rangle| \leq 1$
for all $x,y$, using \eqref{eqn:xox},
\begin{align*}
D
& = \sum_{x \in X} \sum_{y \in Y} M_{x,y} \langle x , y \rangle  \geq \sum_{x \in X} \sum_{y \in Y} ( \langle x , y \rangle )^2 
= \sum_{y \in Y} \sum_{x \in X} \langle y, (x \otimes x) y\rangle  = \frac{N}{k} \sum_{y \in Y}  \langle y, y\rangle
  = \frac{N^2}{k}. \end{align*}
\end{proof}

\subsection{Applications}
\subsubsection{Explicit examples}
\label{sec:PS}
Here we prove Theorem~\ref{thm:projective} and Theorem~\ref{thm:projint}.
\begin{proof}[Proof of Theorem~\ref{thm:projective}]
It is well known that the VC dimension of $A$ is $d$,
but we provide a brief explanation.
The VC dimension is at least $d$ 
by considering any set of $d$ independent points
(i.e.\ so that no strict subset of it spans it).
The VC dimension is at most $d$ since every
set of $d+1$ points is dependent in a $d$ dimensional space.

The lower bound on the sign rank 
follows immediately from Theorem~\ref{thm:SpecVsSignRank},
and the following known bound on the spectral 
gap of these matrices.
\begin{lemma}
\label{lem:normVold}
If $B$ is the boolean version of $A$ then
$$\frac{\sigma_2(B)}{\Delta} = 
\frac{n^{\frac{d-1}{2}}(n-1)}{n^d -1} 
\leq N_{n,d}^{-\frac{1}{2}+\frac{1}{2d}}.$$
\end{lemma}
The proof is so short that we include it here.
\begin{proof}
We use the following two known properties 
(see, e.g.,\ \cite{zbMATH01382771}) of projective spaces.
Both the number of distinct hyperplanes through a point
and the number of distinct points on a hyperplane are $N_{n,d-1}$.
The number of hyperplanes through two distinct points is
$N_{n,d-2}$.

The first property implies that
$A$ is $\Delta = N_{n,d-1}$ regular.
These properties also imply
$$B B^T = \left(N_{n,d-1}-N_{n,d-2} \right) I + N_{n,d-2} J = 
 n^{d-1} I + N_{n,d-2} J,$$
 where $J$ is the all $1$ matrix.
Therefore, all singular values except the maximum one
are $n^{\frac{d-1}{2}}$.
\end{proof}
\end{proof}

{
\begin{proof}[Proof of Theorem~\ref{thm:projint}]
We first show that $R$ is indeed a maximum class of VC dimension $2$.
The VC dimension of $R$ is $2$: It is at least $2$ because $R$ contains the set of lines
whose VC dimension is $2$. It is at most $2$ because no three points $p_1,p_2,p_3$ are shattered. Indeed if they all belong to a line $\ell$ then without loss of generality according to the  order of $\ell$ we have $p_1<p_2<p_3$ which implies that the pattern $101$ is missing. Otherwise, they are not co-linear and the pattern $111$ is missing.

To see that $R$ is a maximum class, note
that there are exactly $N+1$ intervals of size at most one (one empty interval
and $N$ singletons).
For each line $\ell\in L$, the number of intervals of size at least two which are subsets of $\ell$
is exactly ${|\ell| \choose 2} = {n+1 \choose 2}$. Since every two distinct lines intersect in exactly one point, it follows that each interval of size at least two is a subset of exactly one line.
It follows that the number of intervals is
\[
1 + N + N\cdot{n+1 \choose 2} = 1 + N + {N \choose 2}.
\]
Thus, $R$ is indeed a maximum class of VC dimension $2$.

Next we show that there exists a choice of a linear order for each line such that the resulting $R$ has sign rank $\Omega(N^{\frac{1}{2}}/\log N)$.
By the proof of Theorem~\ref{t42}, case $d=2$, there is a choice of a subset for each line such that the resulting $N$ subsets form a class of sign rank $\Omega(N^{\frac{1}{2}}/\log N)$.
We can therefore pick the linear orders in such a way that each of these $N$ subsets forms an interval, and the resulting maximum class 
(of all possible intervals with respect to these orders)
has sign rank at least as large as $\Omega(N^{\frac{1}{2}}/\log N)$.
\end{proof}
}

\subsubsection{Computing the sign rank}
\label{sec:approxproof}
In this section
we describe an efficient algorithm that approximates the sign rank
(Theorem~\ref{t43}).

The algorithm uses the following notion.
Let $V$ be a set. 
A pair $v,u \in V$ is {\em crossed} by a vector $c \in\{\pm 1\}^V$ if $c(v) \neq c(u)$. 
Let $T$ be a tree with vertex set $V = [N]$ and edge set $E$.
Let $S$ be a $V \times [N]$ sign matrix.
The {\em stabbing number} of $T$ in $S$
is the largest number of edges in $T$ that are crossed by the same column of $S$.
For example, if $T$ is a path then $T$ defines a linear order (permutation) on $V$
and the stabbing number is the largest number
of sign changes among all columns with respect to this order.

Welzl~\cite{We} gave an efficient algorithm for computing
a path $T$ with a low stabbing number for matrices $S$ 
with VC dimension $d$. 
%NA
The analysis of the algorithm can be improved by a logarithmic
factor using
a result of~\cite{Ha}.
\begin{theorem}[\cite{We,Ha}]\label{t44}
There exists a polynomial time algorithm such 
that given a $V\times [N]$ sign matrix
$S$ with $|V|=N$, outputs a path on $V$ 
with stabbing number at most $200N^{1-1/d}$
where $d = VC(S)$.
\end{theorem}
%NA
For completeness, and since  to the best of our knowledge
no explicit proof of this theorem appears in print,
we provide a description and analysis of the algorithm.
We assume without loss of generality that
the rows of $S$ are pairwise distinct.

We start by handling the case\footnote{This analysis
also provides an alternative proof for Lemma~\ref{lem:VCone}.} $d=1$.
In this case, we directly output a tree that is a path (i.e., a linear order on $V$).
If $d=1$, then Claim~\ref{clm:ID} implies
that there is a column with at most $2$ sign changes with respect to any order
on $V$.
The algorithm first finds by recursion a path $T$ for the matrix obtained from $S$ by
removing this column, and outputs the same path $T$ for the matrix $S$ as well.
By induction, the resulting path has stabbing number at most $2$
(when there is a single column the stabbing number can be made $1$).

For $d>1$, 
the algorithm constructs a sequence of $N$ forests
$F_0,F_1,\ldots,F_{N-1}$ over the same vertex set $V$.
The forest $F_i$ has exactly $i$ edges,
and is defined by greedily adding an edge $e_i$ to $F_{i-1}$.
As we prove below, the tree $F_{N-1}$ has a stabbing number
at most $100N^{1-1/d}$.
The tree $F_{N-1}$ is transformed to a path $T$ as follows.
Let $v_1,v_2,\ldots,v_{2N-1}$ be an eulerian path in the graph obtained
by doubling every edge in $F_{N-1}$.
This path traverses each edge of $F_{N-1}$ exactly twice.
Let $S'$ be the matrix with $2N-1$ rows and $N$ columns
obtained from $S$ be putting row $v_i$ in $S$ as row $i$, for $i \in [2N-1]$.
The number of sign changes in
each column in $S'$ is at most $2 \cdot 100N^{1-1/d}$.
Finally, let $T$ be the path obtained from the eulerian path
by leaving a single copy of each row of $S$.
Since deleting rows from $S'$ cannot increase the number of sign changes,
the path $T$ is as stated.

The edge $e_i$ is chosen as follows.
The algorithm maintains a probability distribution $p_i$ on $[N]$.
The weight $w_i(e)$ of the pair $e = \{v,u\}$ is the probability mass of the columns 
$e$ crosses, that is, $w_i(e) = p_i(\{ j \in [N] : S_{u,j} \neq S_{v,j} \})$.
The algorithm chooses $e_i$ as an edge with minimum $w_i$-weight among all edges
that are not in $F_{i-1}$ and do not close a cycle in $F_{i-1}$.

The distributions $p_1,\ldots,p_{N}$ are chosen iteratively as follows.
The first distribution $p_1$ is the uniform distribution on $[N]$.
The distribution $p_{i+1}$ is 
obtained from $p_{i}$ by doubling the relative mass of each column
that is crossed by $e_{i}$.
That is, let $x_{i} = w_{i}(e_{i})$,
and for every column $j$ that is crossed by $e_{i}$ define
$p_{i+1}(j)=\frac{2p_{i}(j)}{1+x_{i}}$, and for every 
other column $j$ define $p_{i+1}(j)=\frac{p_{i}(j)}{1+x_{i}}$.

This algorithm clearly produces a tree on $V$,
and the running time is indeed polynomial in $N$.
It remains to prove correctness.
We claim that each column is crossed by at most $O(N^{1-1/d})$ edges in $T$.
To see this, let $j$ be a column in $S$, and let $k$
be the number of edges crossing $j$.
It follows that
$$p_{N}(j) = \frac{1}{N} \cdot 2^k \cdot \frac{1}{(1+x_1)(1+x_2)\ldots(1+x_{N-1})}.$$
To upper bound $k$, we use the following claim.
\begin{claim}
\label{clm:Has}
For every $i$ we have $x_i\leq 4e^2(N-i)^{-1/d}$.
\end{claim}
The claim completes the proof of Theorem~\ref{t44}:
Since $p_{N}(j)\leq 1$ and $d > 1$, 
\begin{align*}
k &\leq \log N +\log\left({1+x_1}\right)+\ldots+\log\left({1+x_{N-1}}\right) \\
   &\leq \log(N) + 2 \left( \ln(1+x_1)+ ... + \ln(1+x_{N-1})\right) \tag{$\forall x:~\log(x)\leq2\ln(x)$} \\ %\tag{$\forall x\geq 0~ \log(1+x)\leq x$}\\
   &\leq \log(N) + 2 (x_1 + ... + x_{N-1})\\
   %\tag{$\forall i~ x_i\leq 4e^2(N-i)^{-1/d}$, see next paragraph}\\
   &\leq \log N + 8e^2N^{1-1/d} %\tag{$N^{-1/d}+\ldots+2^{-1/d}\leq 2N^{1-1/d}$}\\
   \leq 100N^{1-1/d}.
\end{align*}

The claim follows from the following theorem of Haussler.
\begin{theorem}[\cite{Ha}]
\label{t200}
Let $p$ be a probability distribution on $[N]$, and let $\eps>0$.
Let $S \in \{\pm 1\}^{V \times [N]}$ be a sign matrix of VC dimension $d$
so that the $p$-distance between every two distinct rows $u,v$
is large:
$$p(\{j \in [N]: S_{v,j} \neq S_{u,j} \}) \geq \eps.$$
Then,
the number of distinct rows in $S$ is at most
$$e (d+1) \left( 2e / \eps \right)^d
\leq \left(4e^2/\eps\right)^d.$$
\end{theorem}

\begin{proof}[Proof of Claim~\ref{clm:Has}]
Haussler's theorem states that if the number of distinct rows is $M$,
then there must be two distinct rows of $p_i$-distance at most
$4e^2M^{-1/d}$.
There are $N-i$ connected components in $F_i$.
Pick $N-i$ rows, one from each component.
Therefore, there are two of these rows whose distance is at most
$4e^2M^{-1/d}=4e^2(N-i)^{-1/d}$. 
Now, observe that the $w_i$-weight of the pair $\{u,v\}$
equals the $p_i$-distance between $u,v$.
Since $e_i$ is chosen to have minimum weight, $x_i\leq 4e^2(N-i)^{-1/d}$
\end{proof}

We now describe the approximation algorithm.
Let $S$ be an $N\times N$ sign matrix of VC dimension $d$.
Run Welzl's algorithm on $S$,
and get a permutation of the rows of $S$ that yield a low stabbing number.
Let $s$ be the maximum number of sign changes
among all columns of $S$ with respect to this permutation.
Output $s+1$ as the approximation to the sign rank of $S$.

We now analyze the approximation ratio. 
By Lemma~\ref{l43} the sign rank of $S$ is at most
$s+1$. Therefore, the approximation factor $\frac{s+1}{\text{sign-rank}(S)}$ is
at least $1$. On the other hand,
Proposition~\ref{prop:dualsignrank} implies that $d\leq\text{sign-rank}(S)$.
Thus, by the guarantee of Welzl's algorithm,
\begin{align*}
\frac{s+1}{\text{sign-rank}(S)} \leq O\left(\frac{N^{1-1/d}}{\text{sign-rank}(S)}\right) 
\leq O\left(\frac{N^{1-1/d}}{d}\right). \end{align*}
This factor is 
maximized for $d = \Theta(\log N)$ and is therefore at most $O( N/ \log N)$. 

{
\subsubsection{Counting VC classes}\label{sec:counting}

Here we prove Theorems~\ref{thm:VCcount} and~\ref{thm:MAXcount}.
It is convenient for both to set
$$f=\sum_{i=0}^d {N \choose i}.$$

\begin{proof}[Proof of Theorem~\ref{thm:VCcount}]
We start with the upper bound.
Enumerate the members of each such class $C$ as follows. Start
with the (lexicographically) first member $c\in C$, call it $c_1$. Assuming
$c_1,c_2, \ldots ,c_i$ have already been chosen, let $c_{i+1}$ be the
member $c$ among the remaining vectors in $C$ whose hamming distance  from the set $\{c_1,\ldots,c_i\}$
% (defined as the minimum distance to some $c_j, j \leq i$) 
is minimum (in case of equalities we take the first one lexicographically).
This gives an enumeration $c_1,\ldots,c_m$ of the members of $C$, and
$m \leq f$.

We now upper bound the number of possible families. There are at most
$2^N$ ways to choose $c_1$. If the distance of $c_{i+1}$ from the
previous sets is $h=h_{i+1}$, then we can determine $c_{i+1}$ by giving
the index $j \leq i$ so that the distance between $c_{i+1}$ and $c_j$ is
$h$, and by giving the symmetric difference of $c_{i+1}$ and $c_j$. There
are less than $m \leq f$ ways  to choose the index, and at most
${n \choose h} <(eN/h)^h$ options for the symmetric
difference. The crucial point is that by Theorem~\ref{t200} 
the number of $i$ for which $h_i \geq D$ 
is less than $e(d+1)(2eN/D)^d$. Hence the number of $i$ for which $h_i$ is between $2^\ell$ and $2^{\ell+1}$ is at most $e(d+1)(2eN/2^\ell)^d$.
This upper bounds $c(N,d)$ by at most
$$
2^N m^f \prod_{\ell} \left((eN/2^{\ell})^{2^{\ell+1}} \right)^{e(d+1)(2eN/2^\ell)^d}
\leq 2^N f^f N^{(O(N))^d}=N^{(O(N))^d}.
$$

%$$
%2^N m^f \prod_{i} ((eN/2^{i})^{2^{i+1}} )^{e(d+1)(2eN/2^i)^d}
%\leq 2^N f^f N^{O(d N^d)}=N^{O(dN^d)}.
%$$

We now present a lower bound on the number of (maximum)
classes with VC dimension $d$.
Take a family $F$ of ${N \choose {d}}/(d+1)$ subsets of $[N]$ of size
$(d+1)$ so that
every subset of size $d$ is contained in exactly one of them. Such
families exist by a recent breakthrough result of Keevash~\cite{Keevash}, 
%NA1
provided the trivial divisibility conditions hold and $N>N_0(d)$.
His proof also gives that there are $N^{(1+o(1)){N \choose d}/(d+1)}$ such families.

Now, construct a class $C$ by taking all subsets of cardinality at most $d-1$,
and for each $(d+1)$-subset in the family $F$ take it and all its
subsets of cardinality $d$ besides one. The VC dimension
of $C$ is indeed $d$. The number of possible $C$s that can be constructed this way is at least the number of families $F$. 
%(in fact every $F$ yields $(d+1)^F$ different $C$'s). 
Therefore, the number of classes of VC dimension $d$ is at least the number of $F$s:       
$$N^{(1+o(1)){N \choose d}/(d+1)}=N^{(\Omega(N/d))^d}.$$
\end{proof}

\begin{proof}[Proof of Theorem~\ref{thm:MAXcount}]
For the upper bound we use the known fact that every maximum class is a connected subgraph of the boolean cube~\cite{Welzl}. Thus, to upper bound the number of maximum classes of VC dimension $d$ it is enough
to upper bound the number of connected subgraphs of the $N$-dimensional cube of size $f$.
It is known (see, e.g., Lemma 2.1 in~\cite{Al3}) 
that the number of connected subgraphs of size $k$ 
in a graph with $m$ vertices and maximum degree $D$ is at most 
$m(eD)^k$.
In our case, plugging $k=f$, $m=2^N$, $D=N$ yields the desired bound $2^N(eN)^f=N^{(1+o(1))f}$.

For the lower bound, note that in the proof of Theorem~\ref{thm:VCcount} the constructed classes were of size $f$, and therefore maximum classes. Therefore, there are at least $N^{(1+o(1)){N \choose d}/(d+1)}$ maximum classes of VC dimension $d$.
\end{proof}
}

%The proof of Theorem \ref{t41} also supplies an upper bound for the 
%number of $N \times N$ matrices with VC dimension $d$, and in fact
%with primal shatter function $O(t^d)$. Indeed, in each such
%matrix one can permute the rows and get a matrix in which the
%number of sign changes in each column is $O(N^{1-1/d})$. The number
%of ways to choose the permutation is $N!$, and then the number of
%ways to choose each column is at most
%$2^{O(N^{1-1/d} \log N)}$.
%This gives that the total number of such matrices  is at most
%$ 2^{O(N^{2-1/d} \log N)}$.
%By the discussion above, this is tight up to the logarithm in the
%exponent for $d=2$, and for counting matrices with primal shatter
%function $O(t^d)$ it is tight up to this logarithm for any 
%integer $d>1$, by the construction using the matrices of
%\cite{ARS}.
%For VC dimension $1$, it is not difficult to show that the 
%correct number is $2^{\Theta(N \log N)}$.

\subsubsection{Counting graphs}
\label{sec:countingGraphs}
\begin{proof}[Proof of Theorem~\ref{thm:Udfree}]
The key observation is that whenever we split the vertices of a $U(d+1)$-free graph into two
disjoint sets  of equal size, the bipartite graph between them
defines a matrix of VC dimension at most $d$.
Hence, the number of such bipartite graphs is at most
$$
T(N,d)=2^{O(N^{2-1/d} \log N)}.
$$
By a known  lemma of Shearer \cite{CFGS}, this implies that the
total number of $U(d+1)$-free graphs on $N$  vertices is less than
$T(N,d)^2=2^{O(N^{2-1/d} \log N)}$.
For completeness, we include the simple details. The lemma we use
is the following.

\begin{lemma}[\cite{CFGS}]
\label{l47} 
Let ${\cal F}$ be a family of vectors in $S_1 \times
S_2 \cdots \times S_n$. Let 
${\cal G}=\{G_1, \ldots, G_m \}$
be a collection of subsets of $[n]$,
%$K=\{1,2, \ldots ,n\}$, 
and suppose
that each element $i \in [n]$ belongs to at least $k$ members of
${\cal G}$. For each $1 \leq i \leq m$, let ${\cal F}_i$ be the set
of all projections of the members of ${\cal F}$ on 
the coordinates in $G_i$. Then
$$
|{\cal F}|^k \leq \prod_{i=1}^m |{\cal F}_i|.
$$
\end{lemma}

In our application, $n={N \choose 2}$ and $S_1 = \ldots = S_n=\{0,1\}$.
The vectors represent graphs on $N$ vertices,
each vector
being the characteristic vector of a graph on $N$ labeled vertices.
The set $[n]$ corresponds to the set of all ${N \choose 2}$ potential edges.
The family $\cal F$ represents all $U(d+1)$-free graphs.
The collection ${\cal G}$ is the set of all complete 
bipartite graphs with $N/2$ vertices in each color class. 
Each edge $i \in [n]$ belongs to at least (in fact a bit more than) half of them, i.e.,
$k \geq m/2$. 
Hence,
$$
|{\cal F}| \leq \left( \prod_{i=1}^m |{\cal F}_i| \right)
^{2/m} \leq \left( ( T(N,d))^m \right)^{2/m},
$$
as desired.
\end{proof}

\section{Concluding remarks and open problems}
We have given explicit examples of $N \times N$ sign matrices with small
VC dimension and large sign rank. However, we have not 
been able to prove that any of them has sign rank exceeding
$N^{1/2}$. Indeed this seems to be the limit of Forster's approach,
even if we do not bound the VC dimension. 
Forster's theorem
shows that the sign rank of any $N \times N$ Hadamard matrix is
at least $N^{1/2}$. It is easy to see that there
are Hadamard matrices of sign rank
significantly smaller than linear in $N$. Indeed, the sign rank of
the $4 \times 4$ signed identity matrix is $3$, and hence the sign
rank of its $k$'th tensor power, which is an $N \times N$ Hadamard
matrix with $N=4^k$,
is at most $3^k=N^{\log 3/\log 4}$ (a similar argument was given by~\cite{DBLP:journals/tcs/ForsterS06} for the Sylvester-Hadamard matrix). It may well be, however, that
some Hadamard matrices have sign rank linear in $N$, as do random
sign matrices, and it will be very interesting to show that this is
the case for some such matrices.
It will also be interesting to decide what is the correct behavior
of the sign rank of the incidence graph of the points and lines of
a projective plane with $N$ points. 
We have seen that it is at least $\Omega(N^{1/4})$
and at most $O(N^{1/2})$.

Using our spectral technique 
we can give many additional explicit examples of
matrices with high sign rank, including ones for which the
matrices not only have VC dimension $2$, but are more restricted
than that (for example, no $3$ columns have more than
$6$ distinct projections). 
%Here is a brief description.
%An $(N,\Delta, \lambda)$-graph is a $\Delta$ regular graph on
%$N$ vertices so that the absolute value of every eigenvalue of
%the graph besides the top one is at most $\lambda$. There are 
%several known constructions of $(N,\Delta,\lambda)$-graphs for which
%$\lambda \leq O(\sqrt {\Delta})$, that do not contain short
%cycles. Any such graph with $\Delta \geq N^{\Omega(1)}$ provides
%an example with sign rank at least $N^{\Omega(1)}$, and if there is
%no cycle of length at most $6$ then in the sign matrix we have at
%most $6$ distinct projections on any set of $3$ columns.

We have shown that the maximum sign rank $f(N,d)$ of an $N \times N$
matrix with VC dimension $d>1$ 
is at most $O(N^{1-1/d})$, and that this is
tight up to a logarithmic factor for $d=2$,
and close to being tight for large $d$.
It seems plausible to conjecture that
$f(N,d)=\tilde{\Theta}(N^{1-1/d})$ for all $d>1$.

We have also showed how to use this upper bound
to get a nontrivial approximation algorithm for the sign rank.
It will be interesting to fully understand the computational complexity
of computing the sign rank.

Finally we note that most of the analysis in this paper can be
extended to deal with $M \times N$  matrices, where $M$ and $N$ are
not necessarily equal, and we restricted the attention here for 
square matrices mainly in order to simplify the presentation.

\section*{Acknowledgements}
We wish to thank Rom Pinchasi, Amir Shpilka, and Avi Wigderson
for helpful discussions and comments.

\bibliographystyle{plain}
\bibliography{srankbib}

\end{document}